\RequirePackage{fix-cm}
\documentclass{amsart} 
\usepackage[utf8]{inputenc}
\usepackage{mathptmx}
\usepackage{latexsym}
\usepackage{amsmath}
\usepackage{amssymb}
\usepackage{amsfonts} 
\usepackage{eucal} 
\usepackage{amsbsy}
\usepackage{amsthm}
\usepackage[all]{xy}
\usepackage[pagebackref=true]{hyperref}
\usepackage{graphicx}
\usepackage{tikz}

\newtheorem{thm}{Theorem}[section]
\newtheorem*{thm*}{Theorem}
\newtheorem{lemma}[thm]{Lemma}
\newtheorem{prop}[thm]{Proposition}
\newtheorem{cor}[thm]{Corollary}
\newtheorem*{cor*}{Corollary}

\theoremstyle{definition}
\newtheorem{defn}[thm]{Definition}

\theoremstyle{remark}
\newtheorem{remark}[thm]{Remark}

\newcommand {\La}    {\ensuremath{\mathcal{L}}}

\newcommand {\Fa}    {\ensuremath{\mbox{$\mathcal{F}$}}}

\newcommand {\mbl}   {\ensuremath{\mathbb{L}}}

\newcommand {\real}  {\ensuremath{\mathbb{R}}}
\newcommand {\intg}  {\ensuremath{\mathbb{Z}}}

\newcommand {\cplx}  {\ensuremath{\mathbb{C}}}
\newcommand {\rat}   {\ensuremath{\mathbb{Q}}}

\newcommand {\ch}    {\ensuremath{\operatorname{ch}}}
\newcommand {\ph}    {\ensuremath{\operatorname{ph}}}

\newcommand {\smlhf} {\ensuremath{\mbox{$\frac{1}{2}$}}}

\newcommand {\Smash} {\ensuremath{\wedge}}

\newcommand {\syml}  {\ensuremath{\mathbb{L}}}

\newcommand {\loc}   {\ensuremath{\operatorname{loc}}}

\newcommand {\pt}    {\ensuremath{\operatorname{pt}}}
\newcommand {\id}    {\ensuremath{\operatorname{id}}}

\newcommand {\BSO}   {\ensuremath{\operatorname{BSO}}}
\newcommand {\BO}   {\ensuremath{\operatorname{BO}}}
\newcommand {\BU}   {\ensuremath{\operatorname{BU}}}
\newcommand {\BSPL}   {\ensuremath{\operatorname{BSPL}}}

\newcommand {\BBSPL}   {\ensuremath{\operatorname{B}\widetilde{\operatorname{SPL}}}}

\newcommand {\MSO}   {\ensuremath{\operatorname{MSO}}}

\newcommand {\MSPL}   {\ensuremath{\operatorname{MSPL}}}

\newcommand {\RMSPL}   {\ensuremath{\widetilde{\operatorname{MSPL}}}}
\newcommand {\MSTOP}   {\ensuremath{\operatorname{MSTOP}}}

\newcommand {\SO}   {{\ensuremath{\operatorname{SO}}}}
\newcommand {\SPL}   {{\ensuremath{\operatorname{SPL}}}}
\newcommand {\PL}   {{\ensuremath{\operatorname{PL}}}}

\newcommand {\STOP}   {{\ensuremath{\operatorname{STOP}}}}

\newcommand {\Witt}   {{\ensuremath{\operatorname{Witt}}}}

\newcommand {\MWITT}   {\ensuremath{\operatorname{MWITT}}}

\newcommand {\Th}   {\ensuremath{\operatorname{Th}}}

\newcommand {\KO}   {{\ensuremath{\operatorname{KO}}}}
\newcommand {\K}   {{\ensuremath{\operatorname{K}}}}
\newcommand {\ko}   {{\ensuremath{\operatorname{ko}}}}

\newcommand {\ism}   {\ensuremath{\intg [\smlhf]}}

\newcommand {\proj}    {\ensuremath{\operatorname{proj}}}

\makeatletter
\@namedef{subjclassname@2020}{%
  \textup{2020} Mathematics Subject Classification}
\makeatother

\begin{document}

%******************** TITLE PAGE *****************

\title[Transfer and the Siegel-Sullivan Orientation]
  {Transfer and the Spectrum-Level Siegel-Sullivan KO-Orientation 
   for Singular Spaces}

\author{Markus Banagl}

\address{Mathematisches Institut, Universit\"at Heidelberg,
  Im Neuenheimer Feld 205, 69120 Heidelberg, Germany}

\email{banagl@mathi.uni-heidelberg.de}

\thanks{This work is funded by a research grant of the
 Deutsche Forschungsgemeinschaft (DFG, German Research Foundation)
 -- Projektnummer 495696766.}

\date{May 2023}


% 55N33   	Intersection homology and cohomology 
% 55R12   	Transfer 
% 57N80   	Stratifications
% 57R20   	Characteristic classes and numbers (manifolds and cell complexes) 
% 55N15         Topological K-theory
% 19L41         Connective K-theory, cobordism
% 57Q20   	Cobordism (in PL topology) 
% 57Q50   	Microbundles and block bundles (in PL topology) 


%********************** ABSTRACT ****************************

\begin{abstract}
Integrally oriented normally nonsingular maps between singular spaces have
associated transfer homomorphisms on KO-homology at odd primes.
We prove that such transfers preserve Siegel-Sullivan orientations,
defined when the singular spaces are compact pseudomanifolds
satisfying the Witt condition, for example pure-dimensional
compact complex algebraic varieties. 
This holds for bundle transfers associated to block bundles with
manifold fibers as well as for Gysin restrictions associated to
normally nonsingular inclusions.
Our method is based on constructing a lift of the Siegel-Sullivan orientation
to a morphism of highly structured ring spectra which
factors through L-theory.
\end{abstract}

\maketitle

%******************** TABLE OF CONTENTS *********************

\tableofcontents

%=================================================================
%=================================================================
%=================================================================

\section{Introduction}

Let $\KO_* (-)$ denote topological $\KO$-homology and
let $M$ be a smooth $n$-dimensional closed oriented manifold.
In his MIT notes \cite{sullivanmit},
Sullivan introduced a class
$\Delta_\SO (M) \in \KO_n (M) \otimes \ism$,
which is an orientation and
plays a fundamental role in studying the $\K$-theory
of manifolds.
For instance, Sullivan showed that topological block bundles
away from $2$ are characterized as
spherical fibrations together with a $\KO [\smlhf]$-orientation.
He went on to point out in \cite{sullivansinginspaces} 
that given a class of oriented
piecewise-linear (PL) pseudomanifolds equipped with a
bordism invariant signature that extends the signature of
manifolds and satisfies Novikov additivity and a product formula,
an analogous procedure (based on suitable transversality results
in the singular context) still works to determine a canonical
orientation in $\KO_* (-)\otimes \ism$.
Goresky and MacPherson's intersection homology allowed for
the construction of such signature invariants when
the pseudomanifolds have only strata of even codimension, or
more generally, if they satisfy the Witt condition introduced
by Siegel in \cite{siegel}.
An oriented PL pseudomanifold is a Witt space, if the middle-perversity,
middle-dimensional rational intersection homology of links of
odd-codimensional strata vanishes.
This class contains all complex algebraic varieties of pure dimension.
The class of Witt spaces is contained in a yet larger class of
pseudomanifolds, introduced in 
\cite{banagl-mem}, \cite{banagl-lcl}, 
that support a bordism invariant signature.
Roughly, these are spaces that admit a Lagrangian subsheaf
in the link cohomology sheaf along strata of odd codimension.
The present paper, however, will focus only on Witt spaces.

Thus Sullivan's general framework implies that
$n$-dimensional closed Witt spaces $X$ have a canonical orientation
$\Delta (X) \in \KO_* (X)\otimes \ism$.
The construction of this element was described in detail by
Siegel in \cite{siegel}, and is recalled in the present paper.
We shall refer to it as the \emph{Siegel-Sullivan orientation} of a
Witt space. In \cite{csw}, Cappell, Shaneson and Weinberger extended this orientation
to an equivariant class for finite group actions
satisfying a weak regularity condition on the fixed point sets.
Under the Pontrjagin character, $\Delta (X)$ is a lift of the Goresky-MacPherson
$L$-class $L_* (X) \in H_* (X;\rat)$.
The latter already contains significant global information on the singular
space $X$ (see \cite{cw2}, \cite{weinberger})
and its concrete computation is correspondingly challenging.
For complex projective varieties $X$ it is often possible to obtain information on $L_* (X)$
by cutting down to transverse intersections of $X$ with
smooth subvarieties using Gysin homomorphisms.
For example it is possible to reduce $L$-class computations for
singular Schubert varieties entirely to signature computations
of explicitly given \emph{algebraic} subvarieties.
This approach, introduced in \cite{banaglnyjm}
and pursued systematically in \cite{banaglwrazidlo}, 
requires a thorough understanding of how characteristic
and orientation classes for singular spaces transform under Gysin restriction.

Let $g: Y \hookrightarrow X$ be an oriented normally nonsingular 
codimension $c$ inclusion of
closed Witt spaces. Thus $Y$ has an open tubular neighborhood in $X$
which is endowed in a stratum preserving manner with the structure of 
an oriented rank $c$ vector bundle.
Since $\SO$-bundles are $\KO [\smlhf]$-oriented, $g$ has an
associated Gysin homomorphism
$g^!: \KO_* (X) \otimes \ism \to \KO_{*-c} (Y) \otimes \ism.$
We prove (Theorem \ref{thm.gysindeltaWitt}): 

\begin{thm*}
Let $g:Y^{n-c} \hookrightarrow X^n$ be an oriented normally nonsingular inclusion
of closed Witt spaces. The $\KO [\smlhf]$-homology Gysin map $g^!$ of $g$
sends the Siegel-Sullivan orientation of $X$ to the Siegel-Sullivan
orientation of $Y$:
\[ g^! \Delta (X) = \Delta (Y). \]
\end{thm*}

An important class of morphisms in algebraic geometry is given
by local complete intersection morphisms (\cite{fultonintth}) $Y\to B$.
By definition, they admit a factorization $Y\to X\to B$, where
$Y\to X$ is a closed regular algebraic embedding and
$X\to B$ is a smooth morphism. The regular embedding has an associated
algebraic normal vector bundle and the smooth morphism has an associated
relative tangent bundle, so that l.c.i. morphisms possess a
virtual tangent bundle.
A parallel topological notion of normally nonsingular map $Y\to B$ has been 
considered by Goresky-MacPherson \cite[5.4.3]{gmih2} and 
Fulton-MacPherson \cite{fultonmacpherson}.
By definition, these admit factorizations $Y\to X\to B$
into a normally nonsingular inclusion $Y\to X$ followed by
a fiber bundle projection $X\to B$ with manifold fiber.
A complete picture should therefore include an understanding of
how the Siegel-Sullivan orientation behaves under Becker-Gottlieb type
bundle transfer (\cite{beckergottlieb}).
We shall thus also consider bundle transfers $\xi^!$ associated to
block bundles $\xi$ over compact Witt spaces $B$ (\cite{casson}).
These do not require a locally trivial projection map $X\to B$,
but merely a decomposition of $X$ into blocks over cells in $B$. 
(A fiber bundle is a special case of a block bundle.)
Thus, let $\xi$ be an oriented PL $F$-block bundle with closed oriented 
$d$-dimensional PL manifold fiber $F$ over a closed Witt base $B$.
Then, since the stable vertical normal block bundle of $\xi$
is $\KO [\smlhf]$-oriented,
there is a block bundle transfer
$\xi^!: \KO_n (B)\otimes \ism \to \KO_{n+d} (X)\otimes \ism.$
We prove (Theorem \ref{thm.transfdeltaWitt}):

\begin{thm*}
If $\xi$ is an oriented PL $F$-block bundle with closed oriented 
PL manifold fiber $F$ over a closed Witt base $B$,
then the Siegel-Sullivan orientations of base and total space $X$
are related under block bundle transfer by
\[ \xi^! \Delta (B) = \Delta (X). \]
\end{thm*}

Rather than using Siegel's original construction of $\Delta$ directly
to prove the above transfer results, we give a new description
of $\Delta$ based on a homotopy theoretic perspective relating
K- to L-theory:
We provide here a lift of the Siegel-Sullivan orientation to a
ring spectrum level morphism
\[ \Delta: \MWITT \longrightarrow \KO [\smlhf], \]
where $\MWITT$ denotes the ring spectrum representing Witt space bordism
theory, constructed as in \cite{blm} via the ad-theories of
Laures and McClure. A particularly important aspect of $\Delta$
for our present purposes is its multiplicativity.
On homotopy groups, $\Delta_*$ sends the
bordism class of a closed Witt space $X^{4k}$ to its signature $\sigma (X)$.
In order to obtain our ring spectrum level description of $\Delta$,
we use results of Land and Nikolaus \cite{landnikolaus} to construct
in Proposition \ref{prop.einftyequivkol} an equivalence of 
highly structured ring spectra
$\KO [\smlhf] \simeq \syml (\real) [\smlhf]$
which maps the element in $\pi_4 (\KO)[\smlhf]$ whose complexification
is the square of the complex Bott element to the signature $1$
element in $\pi_4 (\syml (\intg))$, where
$\syml (R)$ denotes the (projective) symmetric algebraic
L-theory spectrum of a ring $R$ with involution, introduced first by Ranicki.
Under this equivalence, the Siegel-Sullivan orientation $\Delta (X)$
of a Witt space corresponds to the $\syml (\rat)$-homology orientation
$[X]_\syml$ of Laures, McClure and the author, which generalizes
Ranicki's $\syml$-homology orientation of manifolds to singular spaces.
This then enables us to use $\syml$-theoretic transfer results established
in \cite{banaglnyjm} and \cite{banaglbundletransfer}.
For a PL $F$-fiber bundle $p:X\to B$ over a PL manifold base $B$,
the transfer formula 
$p^! [B]_\mathbb{L} = [X]_\mathbb{L} \in \syml (\intg)_{n+d} (X)$
was stated by L\"uck and Ranicki in \cite{lueckranicki}.

From the analytic viewpoint, Sullivan's orientation
$\Delta_\SO (M)$ is for a (closed, oriented) Riemannian manifold $M$ closely related 
to the class of the signature operator 
in Kasparov's model $\K_* (M) = \operatorname{KK}_* (C(M),\cplx)$
of the K-homology of $M$, see for example
\cite{rosenbergweinberger} and \cite[Prop. 8.3]{landnikolausschlichting}.
Modulo $2$-power torsion, the two classes differ by certain powers of $2$.
For smoothly stratified Witt spaces $X$ equipped with an
incomplete iterated edge metric on the regular part,
Albin, Leichtnam, Mazzeo and Piazza construct in \cite{almpsigpack} 
a signature operator
$\eth_{\operatorname{sign}}$ and a $\K$-homology class
$[\eth_{\operatorname{sign}}] \in \K_* (X)$.
Again, it is possible to go well beyond Witt spaces: The
topological cohomology theory of \cite{banagl-mem} and the analytic $L^2$ de Rham
theory have been treated from a common perspective in \cite{ablmp}.

In view of the algebraic results of \cite{banaglnyjm} and \cite{banaglwrazidlo},
the conclusions of the present paper are also relevant in the context of a
question raised by Jörg Schürmann in \cite{schuermannmsri}:
Is the Siegel-Sullivan orientation $\Delta (X)$ of a pure-dimensional
compact complex algebraic variety $X$ the image of
the intersection homology (mixed) Hodge module on $X$
under the motivic Hodge Chern class transformation
$\operatorname{MHC}_1: K_0 (\operatorname{MHM}(X)) 
\to K^{\operatorname{coh}}_0 (X)$ of Brasselet-Schürmann-Yokura
\cite{bsy}, followed by the $K$-theoretical Riemann-Roch transformation
of Baum, Fulton and MacPherson?

The above material is developed as follows:
Section \ref{sec.multidentkol} collects relevant homotopy theoretic information on
$\KO$ and $\syml$, and constructs the particular equivalence between
$\KO$ and $\syml$ away from $2$ used throughout the rest of the paper.
The orientations of Sullivan and Ranicki for the nonsingular case
are reviewed in Section \ref{sec.orientsullivanranicki}.
Section \ref{sec.classicalsiegelsullivan} sketches the classical construction of the
Siegel-Sullivan orientation for singular Witt spaces given in \cite{siegel}.
Our ring spectrum level construction of the Siegel-Sullivan orientation
is the focus of Section \ref{sec.ringspeclevsiegelsullivan}.
An immediate application of this construction is a proof
of cartesian multiplicativity of the Siegel-Sullivan orientation
(Theorem \ref{thm.siegelsullivcartesianmult}) in Section
\ref{sec.siegelsullivanfullmultiplicativity}.
In Section \ref{sec.bundletransfersiegelsullivan},
we proceed to apply our multiplicative spectrum level construction
in establishing the normally nonsingular block bundle transfer result,
while normally nonsingular Gysin restrictions of the Siegel-Sullivan
orientation are the subject of the final Section \ref{sec.gysinsiegelsullivan}.\\

\textbf{Acknowledgements.}
We thank Jörg Schürmann for interesting discussions on the subject
matter on this paper.
We especially express our gratitude to Markus Land for 
discussions on the relation of K- and L-theory away from $2$.

\section{Multiplicative Identification of KO- and L-theory away from $2$}
\label{sec.multidentkol}

Our method is based on finding a multiplicative equivalence
$\syml [\smlhf] \simeq \KO [\smlhf]$ that yields the Sullivan
orientation when precomposed with Ranicki's orientation 
$\MSPL \to \syml \to \syml [\smlhf]$. We describe such an equivalence  
in the present section, based on an equivalence of highly structured
ring spectra obtained by Land and Nikolaus in \cite{landnikolaus}.

Let $R$ be a commutative unital ring with involution
and let $\syml (R)$ denote the (projective) symmetric algebraic
L-theory spectrum of $R$, introduced first by Ranicki.
(See e.g. \cite{ranickialtm}; 
there is no need for our notation to distinguish between the
symmetric and the quadratic L-theory spectrum, since the latter will not be 
used in the present paper.)
The only instances of $R$ used in this paper are the ring of integers and the
fields of rational, real or complex numbers.
Since $\widetilde{K}_0 (R)$ vanishes in these cases, there is no difference
between the free and the projective L-theory.
The involution on $\intg, \rat$ and $\real$ is taken to be the
trivial involution, while we consider $\cplx$ to be endowed with the complex conjugation
involution.
If $R=\intg,$ we shall briefly write $\syml = \syml (\intg)$.
The spectrum $\syml (R)$ is a ring spectrum, and
a morphism $R\to S$ of commutative rings with involution induces 
a morphism $\syml (R) \to \syml (S)$. 
For the inclusions $\intg \subset \rat \subset \real \subset \cplx$, one
obtains morphisms
\[ \syml(\intg) \longrightarrow 
   \syml(\rat) \longrightarrow 
   \syml(\real) \longrightarrow 
   \syml(\cplx). \]
The multiplicative symmetric Poincar\'e 
ad-theory of Laures and McClure \cite{lauresmcclure} shows that
these are morphisms of ring spectra.
Their lax symmetric monoidal functor also shows that $\syml (R)$ 
for a commutative ring $R$ can be realized as a commutative
symmetric ring spectrum. Thus $\syml (R)$ is equivalent to
an $\mathbb{E}_\infty$-ring spectrum, that is, the multiplication is
commutative and associative not just up to homotopy, but up to 
coherent systems of homotopies.
If $E$ is a ring spectrum and $A$ a subring of $\rat$, then the localized
spectrum $E_A$ is endowed with a unique (up to ring equivalence) ring
structure such that the localization morphism $E \to E_A$ is a ring morphism.
Ring morphisms $E\to F$ localize to ring morphisms $E_A \to F_A$.
If $A=\intg [\smlhf],$ we will write $E[\smlhf]$ for $E_A$.
Thus there are canonical morphisms of ring spectra
\[ \syml(\intg)[\smlhf] \longrightarrow 
   \syml(\rat)[\smlhf] \longrightarrow 
   \syml(\real)[\smlhf] \longrightarrow 
   \syml(\cplx)[\smlhf]. \]
The first two morphisms are equivalences, the last one is not.   
Indeed, the homotopy ring of $\syml (\intg)[\smlhf]$ is given by
$\syml_* (\intg)[\smlhf] = \intg [\smlhf] [x^{\pm 1}],$
where $x\in \syml_4 (\intg)$ is the signature $1$ element.
(The degree $1$ torsion element in $\syml_* (\intg)$ is removed by
inverting $2$.)
The homotopy ring of $\syml (\real)$ is given by
$\syml_* (\real) = \intg [x^{\pm 1}],$
where $x\in \syml_4 (\real)$ denotes the image of the class
$x\in \syml_4 (\intg)$.
The homotopy groups $\syml_i (\rat)$ vanish in degrees $i$ not
divisible by $4$. For $i=4k,$ they contain an infinitely generated
amount of $2$-primary torsion,
\[ \syml_{4k} (\rat) = \syml_{4k} (\real) \oplus \bigoplus_{p \text{ prime}}
    \syml_{4k} (\mathbb{F}_p) =
    \intg \oplus (\intg/_2)^\infty \oplus (\intg/_4)^\infty, \]
where $\mathbb{F}_p$ denotes the finite field with $p$ elements.
The signature homomorphism $\syml_{4k} (\rat)\to \intg$ provides the
unique splitting for the unique ring homomorphism
$\intg \to \syml_{4k} (\rat)$.     
The infinitely generated amount of $2$-primary torsion is then removed
by inverting $2$.    
Via the above canonical multiplicative maps, we shall identify
$\syml(\intg)[\smlhf],$
$\syml(\rat)[\smlhf]$ and 
$\syml(\real)[\smlhf]$ as $\mathbb{E}_\infty$-ring spectra.
The homotopy ring of $\syml (\cplx)$ (with conjugation involution on $\cplx$) 
is $\syml_* (\cplx) = \intg [b^{\pm 1}],$
where $b$ has degree $2$.
The spectrum $\syml (\cplx)$ is $2$-periodic.
On homotopy rings, the map $\syml (\real) \to \syml (\cplx)$
induces the map $\intg [x^{\pm 1}] \to \intg [b^{\pm 1}],$ $x\mapsto b^2$.

Let $\KO$ denote the $8$-periodic ring spectrum representing real K-theory
and $\K$ the $2$-periodic ring spectrum representing complex K-theory.
The homotopy ring of $\K$ is $\pi_* (\K)=\intg [\beta^{\pm 1}],$
where $\beta$ is the complex Bott element in degree $2$, i.e.
$\beta$ is represented by the reduced canonical complex line bundle
$H-1 \in \widetilde{\K}^0 (S^2)$.
The complexification $c: \BO \to \BU$ can be lifted to a morphism of
spectra $c:\KO \to \K$ (\cite[p. 360, Lemma VI.3.3]{rudyak}).
On $\pi_4$, $c$ induces multiplication by $2$,
$c_* =2: \pi_4 (\KO) = \intg \to \intg = \pi_4 (\K)$.
Thus there does not exist an element in $\pi_4 (\KO)$ that maps to
$\beta^2$. But after inverting $2$, such an element exists.
Let $a\in \pi_4 (\KO)[\smlhf]$ be the element whose complexification
is $\beta^2$.
The localization $\KO [\smlhf]$ is a $4$-periodic ring spectrum
with homotopy ring
$\pi_* (\KO)[\smlhf] =\intg [\smlhf] [a^{\pm 1}].$

Taylor and Williams showed in \cite[Theorem A]{taylorwilliams} that there is
an equivalence
$\syml (\intg) [\smlhf] \simeq \KO [\smlhf]$ of spectra.
In \cite{rosenberg}, Rosenberg asserts that these spectra are 
equivalent as homotopy ring spectra.
Further arguments in this direction are supplied by Lurie
\cite{lurie}, who proves that these spectra are quivalent as homotopy ring
spectra. Land and Nikolaus \cite[p. 550]{landnikolaus}
construct an equivalence of $\mathbb{E}_\infty$-ring spectra
$\KO [\smlhf] \simeq \syml (\intg) [\smlhf].$

\begin{prop} \label{prop.einftyequivkol}
There exists an equivalence of $\mathbb{E}_\infty$-ring spectra
\[ \kappa: \KO [\smlhf] \stackrel{\simeq}{\longrightarrow} 
   \syml (\real) [\smlhf] \]
which induces the ring isomorphism
\[ \intg [\smlhf] [a^{\pm 1}] \longrightarrow
   \intg [\smlhf] [x^{\pm 1}],~
   a \mapsto x \]
on homotopy rings.   
\end{prop}
\begin{proof}
We are indebted to Markus Land for communication on the following argument.
Let $\tau_\real: \KO [\smlhf] \to \syml \real [\smlhf]$
be the equivalence of $\mathbb{E}_\infty$-ring spectra
given in \cite[Cor. 5.4]{landnikolaus}.
This equivalence is related to a complex version
$\tau_\cplx: \K [\smlhf] \to \syml \cplx [\smlhf]$
by the commutative diagram 
\[ \xymatrix{
\KO [\smlhf] \ar[d]_c \ar[r]^{\tau_\real} & \syml (\real)[\smlhf] \ar[d] \\
\K [\smlhf] \ar[r]_{\tau_\cplx} & \syml (\cplx)[\smlhf].
} \]
(There is no integral version of $\tau_\cplx$ on the periodic
spectra, although there \emph{is} an integral version 
$\operatorname{k} \to \ell \cplx$ on connective spectra, which induces $\tau_\cplx$
on the periodic spectra after inverting $2$.)
By \cite[Lemma 4.9]{landnikolaus},
$\tau_{\cplx *}: \pi_2 (\K)[\smlhf] \to \pi_2 (\syml \cplx)[\smlhf]$
maps $\beta \mapsto 2b$.
Since $\tau_\cplx$ is a map of $\mathbb{E}_\infty$-ring spectra, it
follows that it sends $\beta^2 \mapsto 4b^2$.
Consequently, the right hand vertical map sends the element
$\tau_{\real *} (a)$ to
\[ \tau_{\cplx *} c_* (a) = \tau_{\cplx *} (\beta^2) = 4b^2. \]
Since $\syml (\real)[\smlhf] \to \syml (\cplx)[\smlhf]$
maps $x \mapsto b^2$,
we deduce that $\tau_{\real *} (a) = 4x$.

Let 
\[ \psi^2: \KO [\smlhf] \longrightarrow \KO [\smlhf] \]
be the stable Adams operation, constructed as a morphism of 
$\mathbb{E}_\infty$-ring spectra
(\cite{daviesstructuresheaf},
\cite[p. 3]{daviesadamsopstmf},
\cite[p. 106]{mayeinfinityring}).
On the homotopy groups $\pi_{4k} \KO [\smlhf] = \intg [\smlhf],$
$\psi^2$ induces $a^k \mapsto 4^k a^k$.
Thus $\psi^2$ induces an isomorphism of homotopy rings,
and is therefore an equivalence.
Composing the inverse of $\psi^2$ with the Land-Nikolaus
equivalence
$\tau_\real: \KO [\smlhf] \to \syml \real [\smlhf]$,
we obtain the desired equivalence $\kappa$ of $\mathbb{E}_\infty$-ring spectra
since the map induced by the composition on $\pi_4$ sends
\[ a \mapsto \mbox{$\frac{1}{4}$} a \stackrel{\tau_{\real *}}{\mapsto} 
    \mbox{$\frac{1}{4}$} (4x)=x. \]
\end{proof}

\section{The Orientations of Sullivan and Ranicki}
\label{sec.orientsullivanranicki}

Let $\MSO, \MSPL$ and $\MSTOP$ denote the Thom spectra
of oriented vector-, PL- and topological bundles.
These are ring spectra and Pontrjagin-Thom isomorphisms
identify their homotopy groups with the bordism groups 
$\Omega^\SO_*, \Omega^\SPL_*, \Omega^\STOP_*$ of
oriented smooth, PL or topological manifolds.
(The topological spectrum $\MSTOP$ will not play an essential
role in what follows, but occasional side remarks will involve it.)
Sullivan obtained in \cite{sullivanmit} a morphism of spectra
\[ \Delta_\SO: \MSO \longrightarrow \KO [\smlhf] \]
such that the induced map on homotopy groups
\[ \Delta_{\SO *}: \Omega^\SO_{4k} = \MSO_{4k} \longrightarrow
   \KO [\smlhf]_{4k} = \intg [\smlhf] \langle a^k \rangle \]
is
\begin{equation} \label{equ.sullorientoncoeffs} 
\Delta_{\SO *} [M^{4k}] = \sigma (M)\cdot a^k, 
\end{equation}
where $\sigma (M)$ denotes the signature of the 
smooth oriented closed manifold $M$, see \cite[pp. 83--85]{madsenmilgram}.   
The Pontrjagin character
$\ph: \KO [\smlhf] \to H\rat [t^{\pm 1}]$
of $\Delta_\SO$ is the inverse of the universal Hirzebruch
$L$-class $L\in H^* (\BSO;\rat)$ up to multiplication with the
stable Thom class $u\in H^0 (\MSO;\intg)$,
\[ \ph (\Delta_\SO) = L^{-1} \cup u \in H^* (\MSO;\rat). \]
The \emph{Sullivan orientation} of a smooth closed $n$-dimensional manifold
$M$ is given by the image
\[ \Delta_\SO (M) = \Delta_{\SO *} [\id_M] \in \KO_n (M)\otimes \intg [\smlhf]  \]
of the bordism class of the identity on $M$ under the
homomorphism
$\Delta_{\SO *}: \Omega^\SO_n (M) \to (\KO [\smlhf])_n (M)$
induced by the spectrum level Sullivan orientation $\Delta_\SO$.
If $M$ has a boundary $\partial M$, then $\Delta_\SO (M)$ is an element in
the relative group $\KO_n (M,\partial M)\otimes \intg [\smlhf]$.

The orientation $\Delta_\SO$ extends canonically to a map
$\Delta_\SPL: \MSPL \to \KO [\smlhf]$ with respect to
the canonical map $\MSO \to \MSPL$ (\cite[Chapter 5.A, D]{madsenmilgram}).
On homotopy groups, the induced map continues to be given by the
signature, i.e. 
\begin{equation} \label{equ.sullorientonplcoeffs} 
\Delta_{\SPL *} [M^{4k}] = \sigma (M)\cdot a^k 
\end{equation}
for an oriented closed PL manifold $M$.
The Pontrjagin character is
\begin{equation} \label{equ.phdeltaspl} 
\ph (\Delta_\SPL) = L^{-1}_\PL \cup u_\PL \in H^* (\MSPL;\rat), 
\end{equation}
where $L_\PL$ is the universal PL $L$-class
$L_\PL \in H^* (\BSPL;\rat)$ and $u_\PL$ the
stable Thom class $u_\PL \in H^0 (\MSPL;\intg)=\intg$
(\cite[Cor. 5.4, p. 102]{madsenmilgram}).
Note that $L_\PL$ restricts to $L$ under the canonical map $\BSO \to \BSPL$.
We recall that this map is a rational equivalence, 
so in particular induces an isomorphism
$H^* (\BSPL;\rat) \stackrel{\simeq}{\longrightarrow}
H^* (\BSO;\rat)$, and this isomorphism identifies Thom's Pontrjagin class
$p_{4i} \in H^{4i} (\BSPL;\rat)$ with the rational reduction 
$p_{4i} \in H^{4i} (\BSO;\rat)$ of the
integral Pontrjagin class  (though the
PL Pontrjagin classes are in general not integral).
For this reason, one commonly identifies $L$ and $L_\PL$ and simply writes
$L$ for it.
An $n$-dimensional closed PL manifold has a
Sullivan orientation
\[ \Delta_\SPL (M) = \Delta_{\SPL *} [\id_M] \in \KO_n (M)\otimes \intg [\smlhf],  \]
where $[\id_M] \in \Omega^\SPL_n (M)$.

\begin{remark}
Randal-Williams observes in \cite[p. 9]{randalfamilysignature} that
this map can be further canonically extended to a map
$\Delta_\STOP: \MSTOP \to \KO [\smlhf]$ with respect to
the canonical forget map $\MSPL \to \MSTOP$, since the fiber
of the latter map is $2$-local.
In particular, given an element $[N] \in \Omega^\STOP_{4k} = \MSTOP_{4k},$
there exists a large integer $i$ such that $2^i N$ is
topologically bordant to a PL manifold $M$, for which
(\ref{equ.sullorientonplcoeffs}) is available. Thus
 \begin{align*}
2^i \Delta_{\STOP *} [N]
&= \Delta_{\STOP *} [M]
   = \Delta_{\SPL *} [M]
    = \sigma (M) \cdot a^k \\
&=  \sigma (2^i N) \cdot a^k
   = 2^i \sigma (N) \cdot a^k.
\end{align*}
This shows that the map induced by
$\Delta_\STOP$ on homotopy groups is again given by the
signature:
\[
\Delta_{\STOP *} [M^{4k}] = \sigma (M)\cdot a^k 
\]
for an oriented closed topological manifold $M$.
\end{remark}

In the present paper, we take an $\syml$-theoretic perspective in order to
approach the orientations of Sullivan and Siegel-Sullivan.
For smooth manifolds $M$, it has been known for a long time that
Ranicki's fundamental class $[M]_\syml$ agrees with Sullivan's
class $\Delta_\SO (M)$
under the appropriate identification of $\KO$- and $\syml$-homology
at odd primes, see for example
Ranicki \cite[p. 15]{ranickialtm}, and 
Weinberger \cite[p. 82]{weinberger}.
Let $L^n (R)$ denote Ranicki's (projective) symmetric $L$-groups of a 
ring $R$ with involution.
In \cite[p. 385, Prop. 15.8]{ranickiats}, Ranicki constructed a morphism of ring spectra
\[ \sigma^*: \MSPL \longrightarrow \syml (\intg) \]
such that the resulting $\syml (\intg)$-homology fundamental class
\[ [M]_\syml := \sigma^* [\id_M] \in \syml (\intg)_n (M)  \]
of a closed oriented PL $n$-dimensional manifold $M$
hits the Mishchenko-Ranicki symmetric signature
\[ \sigma^* (M) = A[M]_\syml \in L^n (\intg [\pi_1 M]) \]
under the assembly map
\[ A: \syml (\intg)_n (M) \longrightarrow L^n (\intg [\pi_1 M]). \]
 (Ranicki extended $\sigma^*$ to a morphism of ring spectra
$\MSTOP \longrightarrow \syml (\intg)$ in \cite[p. 290]{ranickitotsurgob}, 
but we shall not require this extension for the purposes of the present paper.)
Technically, we will work with $\sigma^*$ as constructed by
Laures and McClure in \cite{lauresmcclure} using ad-theories.
Their incarnation of $\sigma^*$ is an $\mathbb{E}_\infty$-ring map
(\cite[1.4]{lauresmcclure}).
The localization morphism $\syml (\intg) \to \syml (\intg)[\smlhf]$
is a morphism of ring spectra. Thus its composition with $\sigma^*$
is a morphism of ring spectra
$\MSPL \to \syml (\intg)[\smlhf]$, which we shall also denote by $\sigma^*$.
By \cite[p. 243]{ranickiatsii}, $\sigma^*$
induces on homotopy groups the map
\[ \sigma^*_{\pt}: \Omega^\SPL_{4k} (\pt) = \MSPL_{4k} (\pt) \longrightarrow
   \syml (\intg) [\smlhf]_{4k} = \intg [\smlhf] \langle x^k \rangle \]
given by
\begin{equation} \label{equ.ranorientoncoeffs}
\sigma^*_{\pt} [M^{4k}] = \sigma (M)\cdot x^k. 
\end{equation}

\begin{remark} \label{rem.symmsigntoordsignformfds}
Let $M$ be a closed oriented PL manifold of dimension $n$.
The constant map $c:M\to \pt$ induces a diagram
\[ \xymatrix{
\Omega^\SPL_n (M) \ar[r]^{\sigma^*} \ar[d]_{c_*} &
 \syml (\intg)_n (M) \ar[r]^A \ar[d]_{c_*} &
   L^n (\intg[\pi_1 M]) \ar[d]_{c_*} \\
\Omega^\SPL_n (\pt) \ar[r]_{\sigma^*_{\pt}} &
 \syml (\intg)_n (\pt) \ar@{=}[r]_A &
   L^n (\intg)   
} \]
which commutes, since the assembly map $A$ is natural.
Together with Equation (\ref{equ.ranorientoncoeffs}),
this diagram shows that the homomorphism 
$L^n (\intg[\pi_1 M]) \to L^n (\intg),$ $n=4k,$
sends the symmetric signature of $M$ to its ordinary signature.
Indeed,
\begin{align*}
c_* \sigma^* (M)
&= c_* A [M]_\syml = c_* A \sigma^* [\id_M] \\
&= \sigma^*_{\pt} c_* [\id_M] = \sigma^*_{\pt} [M]
  = \sigma (M) x^k.
\end{align*}
The analogous fact holds also for singular pseudomanifolds that satisfy the
Witt condition and will be used later to compute the behavior of the 
ring-spectrum level Siegel-Sullivan orientation defined in the present
paper on homotopy groups
(Proposition \ref{prop.speclevelsiegelsullonhtpygrps}).
\end{remark}

Ranicki  \cite[p. 390f]{ranickiats} introduced an $\syml$-theoretic Thom class 
\[ u_\syml (\alpha) \in \widetilde{\syml}^m (\Th (\alpha)) \]
%(\cite[pp. 290, 291]{ranickitotsurgob} for top bundles) 
for oriented rank $m$ PL microbundles
(or PL $(\real^m,0)$-bundles) $\alpha$ as follows:
The classifying map $X\to \BSPL_m$ of $\alpha$
(where $X$ is the base space)
is covered by a bundle map from $\alpha$ to the
universal oriented PL microbundle.
The induced map on Thom spaces yields a class 
\[ u_\SPL (\alpha) \in \RMSPL^m (\Th (\alpha)), \]
the \emph{Thom class} of $\alpha$ in oriented PL cobordism.
It is indeed an $\MSPL$-orientation of $\alpha$
in Dold's sense.
Ranicki then defines
\[ u_\syml (\alpha) := \sigma^* (u_\SPL (\alpha)). \]
Since $\sigma^*: \MSPL \to \syml (\intg)$ is multiplicative,
the element $u_\syml (\alpha)$ is indeed an $\syml$-orientation of
$\alpha$.
This can also be carried out for stable PL bundles $\alpha$.
For the universal stable PL bundle there is thus a canonical
$\syml$-orientation $u_\syml \in \widetilde{\syml}^0 (\MSPL)$.
Since the stable Thom class 
$u_\SPL \in \RMSPL^0 (\MSPL)$ 
of the universal stable PL bundle is given by the identity
$\MSPL \to \MSPL,$ we have $u_\syml = \sigma^*$.
The morphism of ring spectra 
$\syml (\intg)\to \syml (\rat)$ induces a homomorphism
\[ \widetilde{\syml (\intg)}^m (\Th (\alpha)) \longrightarrow 
   \widetilde{\syml (\rat)}^m (\Th (\alpha)). \]
We denote the image of $u_\syml (\alpha)$ under this map again by
$u_\syml (\alpha).$ Furthermore,
the images of these elements in the bottom row of the localization square
\[ \xymatrix{
(\widetilde{\syml \intg})^m (\Th (\alpha)) \ar[r] \ar[d] &
   (\widetilde{\syml \rat})^m (\Th (\alpha)) \ar[d] \\
(\widetilde{\syml \intg [\smlhf]})^m (\Th (\alpha)) \ar@{=}[r] & 
   (\widetilde{\syml \rat [\smlhf]})^m (\Th (\alpha))   
} \]
will also be written as $u_\syml (\alpha).$
Since all maps in the square are induced by morphisms of ring spectra,
all these image elements are again orientations of $\alpha$.

Comparing Equations (\ref{equ.sullorientonplcoeffs}) 
and (\ref{equ.ranorientoncoeffs}), we find that the diagram
\[ \xymatrix{
\MSPL_{4k} (\pt) \ar[rd]_{\Delta_{\SPL *}} \ar[r]^{\sigma^*_{\pt}} 
  & \syml (\intg)[\smlhf]_{4k} \\
 & \KO [\smlhf]_{4k} \ar[u]_\simeq^{\kappa}
} \]
commutes. More is true:
\begin{prop} \label{prop.kappadeltasplisranickior}
The composition 
\[ \MSPL \stackrel{\Delta_\SPL}{\longrightarrow}
   \KO [\smlhf] \stackrel{\kappa}{\longrightarrow}
   \syml (\intg)[\smlhf] \]
of Sullivan's orientation $\Delta_\SPL$ with
the ring equivalence $\kappa$ from Proposition
\ref{prop.einftyequivkol} is homotopic to Ranicki's orientation $\sigma^*$.
\end{prop}
\begin{proof}
We start \emph{at} the prime $2$ with the cohomology class
$L\in H^{4*} (\syml; \intg_{(2)})$ constructed by Taylor and Williams
in \cite{taylorwilliams}.
This yields a specific homotopy class
\[ L: \syml (\intg)_{(2)} \longrightarrow
    \bigoplus_{i\in \intg} H\intg_{(2)} [4i]. \]
The pullback of this class under Ranicki's orientation $\sigma^*$
corresponds under the Thom isomorphism to the
Morgan-Sullivan class $\La \in H^* (\BSPL; \intg_{(2)})$ of
\cite{morgansullivan}.
Rationally, $\La$ becomes the inverse 
$L^{-1} \in H^* (\BSPL; \rat)$ of the Thom-Hirzebruch $L$-class.
The composition
\[ \syml (\real)_{(2)} \longrightarrow
   \syml (\intg)_{(2)} \stackrel{L}{\longrightarrow}
   \bigoplus_{i\in \intg} H\intg_{(2)} [4i] \]
is an equivalence.   
(The individual arrows are not --- the discrepancy is the de Rham invariant.)   
Rationally (i.e. inverting $2$), this gives an equivalence
\[ \syml (\real)_{(0)} = \syml (\intg)_{(0)} \stackrel{\simeq}{\longrightarrow}
   \bigoplus_{i\in \intg} H\rat [4i]. \]
The map
\[ \syml^* (\MSPL) \longrightarrow H^* (\MSPL;\rat)  \]
induced by the composition
\[ \syml (\intg) \stackrel{\operatorname{loc}}{\longrightarrow}
   \syml (\intg)_{(0)} = \syml (\real)_{(0)} \stackrel{\simeq}{\longrightarrow}
   \bigoplus_{i\in \intg} H\rat [4i] \]
thus sends the universal stable $\syml$-orientation
$u_\syml = \sigma^* \in \widetilde{\syml}^0 (\MSPL)$
to 
\[ L^{-1}_\PL \cup u_\PL \in H^* (\MSPL;\rat), \] 
see also
\cite[Remark 16.2, p. 176]{ranickialtm}.
It sends the signature $1$ element $x^k \in \syml_{4k} (\intg)$ to
$1 \in \pi_{4k} (\bigoplus_{i\in \intg} H\rat [4i])$.   
Consider the diagram
\begin{equation} \label{dia.kophsymltw} 
\xymatrix@C=50pt{
\KO [\smlhf] \ar[d]_\kappa^\simeq \ar[r]^{\operatorname{loc}} &
  \KO_{(0)} \ar[d]_{\kappa_{(0)}}^\simeq \ar[r]^{\ph}_\simeq &
   \bigoplus_{i\in \intg} H\rat [4i] \ar@{=}[d] \\
\syml [\smlhf] \ar[r]^{\operatorname{loc}} &
  \syml_{(0)} \ar[r]^\simeq &
   \bigoplus_{i\in \intg} H\rat [4i] .    
} 
\end{equation}
The left hand localization square commutes for general reasons.
The right hand square commutes up to homotopy as well:
Since the involved spectra are graded Eilenberg-MacLane spectra
of graded $\rat$-vector spaces,
it suffices to check that the induced square on homotopy rings
commutes. The vertical isomorphism induced by $\kappa_{(0)}$
sends, according to its very construction, the generator 
$a^k \in \pi_{4k} (\KO_{(0)})$ to $x^k \in \pi_{4k} (\syml_{(0)})$,
which in turn maps to $1 \in \pi_{4k} (\bigoplus_{i\in \intg} H\rat [4i])$.  
As for the Pontrjagin character,
\[  \ph (a^k) = (\ch \circ c)(a^k) = \ch (\beta^{2k})
    = \ch (\beta)^{2k} = 1.  \]
Hence the right hand square commutes up to homotopy.
During the course of the above argument, we have drawn upon several relevant
remarks made by Randal-Williams in \cite{randalfamilysignature}.    
Now evaluate the above diagram on $\MSPL$:
\[ \xymatrix@C=50pt{
(\KO [\smlhf])^0 (\MSPL) 
  \ar[d]_\kappa^\simeq \ar[r]^{\operatorname{loc}} &
  (\KO_{(0)})^0 (\MSPL) 
      \ar[d]_{\kappa_{(0)}}^\simeq \ar[r]^{\ph}_\simeq &
   \bigoplus_{i\in \intg} H^{4i} (\MSPL; \rat) \ar@{=}[d] \\
(\syml [\smlhf])^0 (\MSPL) \ar[r]^{\operatorname{loc}} &
  (\syml_{(0)})^0 (\MSPL) \ar[r]^\simeq &
   \bigoplus_{i\in \intg} H^{4i} (\MSPL; \rat).    
} \]
We shall show that the elements
\[ \Delta_\SPL,~ \kappa^{-1} \circ \sigma^* \in
     (\KO [\smlhf])^0 (\MSPL)   \]
are equal.
According to (\ref{equ.phdeltaspl}), 
$\ph (\Delta_\SPL) = L^{-1}_\PL \cup u_\PL$.
By the commutativity of the diagram,
the Pontrjagin character $\ph (\kappa^{-1} \circ \sigma^*)$, 
given by mapping the element horizontally, 
can alternatively be calculated by first mapping down vertically,
and then mapping to the right horizontally.
Mapping down via $\kappa$ yields $\sigma^*$,
which is then mapped to $L^{-1}_\PL \cup u_\PL$ as discussed above.
It follows that the two elements have the same Pontrjagin character,
\[ \ph (\Delta_\SPL) = \ph (\kappa^{-1} \circ \sigma^*). \]
Now the element $\Delta_\SPL$ is characterized by its
Pontrjagin character, \cite[p. 115]{madsenmilgram}.
(Madsen and Milgram show that the Pontrjagin character
$\ph: \widetilde{\KO} (\MSPL_{(p)}) \to H^* (\MSPL;\rat)$ 
is injective at every odd prime $p$, \cite[Cor. 5.25]{madsenmilgram}.)
It follows that $\kappa^{-1} \sigma^*$ is homotopic to
$\Delta_\SPL$.
\end{proof}

\begin{cor}
The Sullivan orientation $\Delta_\SPL: \MSPL \to \KO [\smlhf]$
is homotopic to a morphism of (homotopy) ring spectra.
\end{cor}

In view of Proposition \ref{prop.kappadeltasplisranickior}, we may thus adopt the 
following convention for the spectrum level Sullivan orientation:
\begin{defn} \label{def.sullorientringmap}
Let 
\[ \Delta: \MSPL \longrightarrow \KO [\smlhf] \]
be the morphism of ring spectra given by the composition
\[ \MSPL \stackrel{\sigma^*}{\longrightarrow} 
  \syml (\intg)[\smlhf] = \syml (\real)[\smlhf] 
   \stackrel{\kappa^{-1}}{\simeq} \KO [\smlhf]  \]
of Ranicki's orientation with a ring equivalence $\kappa^{-1}$
inverse to the ring equivalence $\kappa$ of
Proposition \ref{prop.einftyequivkol}.
\end{defn}

\begin{defn} \label{def.deltaofplmicrobundle}
For an oriented rank $m$ PL microbundle
(or PL $(\real^m,0)$-bundle) $\alpha$, let
\[ \Delta (\alpha) \in \widetilde{\KO}^m (\Th (\alpha)) [\smlhf] \]
be the image
\[ \Delta (\alpha) = \Delta_* (u_\SPL (\alpha)) \]
of the $\MSPL$-orientation under
\[ \RMSPL^m (\Th (\alpha)) \stackrel{\Delta_*}{\longrightarrow}
    \widetilde{\KO}^m (\Th (\alpha)) [\smlhf]. \]
This is indeed a $\KO [\smlhf]$-orientation of $\alpha$,
since $\Delta$ is multiplicative 
(\cite[p. 305, Prop. V.1.6]{rudyak}).
\end{defn}
From this perspective, it is immediate that $\kappa$ aligns $\Delta (\alpha)$
and Ranicki's Thom class $u_\syml (\alpha)$:
\begin{lemma} \label{lem.kappadeltaisranickiu}
Let $\alpha: X \to \BSPL (m)$ be
an oriented PL microbundle of rank $m$.
Then the isomorphism
\[ \kappa_*: \widetilde{\KO}^m (\Th (\alpha))\otimes \intg [\smlhf] 
   \stackrel{\simeq}{\longrightarrow} 
   \widetilde{\syml}^m (\Th (\alpha)) \otimes \intg [\smlhf] \]
maps $\Delta (\alpha)$ to $u_\syml (\alpha)$.
\end{lemma}
\begin{proof}
According to Ranicki's definition,
$u_\mbl (\alpha) = \sigma^* (u_\SPL (\alpha)).$ 
Consequently,
\[
\kappa_* \Delta (\alpha) = \kappa_* \Delta_* (u_\SPL (\alpha))
 = \kappa_* \kappa^{-1}_* \sigma^* (u_\SPL (\alpha))
    = \sigma^* u_\SPL (\alpha) 
= u_\mbl (\alpha).    
\]
\end{proof}

Given an oriented vector or PL bundle $\alpha$ of rank $m$, let
$u_\intg (\alpha) \in \widetilde{H}^m (\Th (\alpha);\intg)$ denote
its integral Thom class and 
$u_\rat (\alpha) \in \widetilde{H}^m (\Th (\alpha);\rat)$
the image of $u_\intg (\alpha)$ under
$\widetilde{H}* (-;\intg)\to \widetilde{H}^* (-;\rat)$.
The methods used to prove 
Proposition \ref{prop.kappadeltasplisranickior} imply readily:
\begin{lemma} \label{lem.phdeltaalpha}
Let $\alpha$ be an oriented PL microbundle over $X$.
Rationally, $\Delta (\alpha)$ is given by
\[ \ph \Delta (\alpha) = L^{-1} (\alpha) \cup u_\rat (\alpha) \in H^* (X;\rat). \]
\end{lemma}
\begin{proof}
One evaluates the commutative diagram (\ref{dia.kophsymltw}) 
on the base space of $\alpha$
and notes that by Lemma \ref{lem.kappadeltaisranickiu},
$\kappa_* \Delta (\alpha) = u_\syml (\alpha_\STOP).$ By commutativity,
$\ph \Delta (\alpha)$ may be computed by mapping $\kappa_* \Delta (\alpha)$ 
along the lower horizontal composition.
As observed in the proof of the proposition, the $\syml$-cohomology Thom class 
is given rationally (i.e. along the lower horizontal composition) 
by the product of the inverse $L$-class with the $H\rat$-cohomology
Thom class $u_\rat$, $L^{-1} (\alpha) \cup u_\rat (\alpha)$.
(See \cite[Remark 16.2]{ranickialtm}; Ranicki writes $-\otimes \rat$ instead
of $\ch_\syml \loc$ and omits cupping with $u_\rat$ in his notation.)
\end{proof}

For the sake of completeness, we also record the case of the trivial bundle:
\begin{lemma}
If $\alpha$ is the trivial rank $4k=m$ PL bundle over a point,
then $\Delta (\alpha) = a^k \in \widetilde{\KO}^{4k} (S^{4k}) [\smlhf]$.
\end{lemma}
\begin{proof}
The Chern character
$\ch: \K^0 (S^{4k}) \to \bigoplus_{i=0}^\infty H^{2i} (S^{4k};\rat)$
is injective and has image
$H^* (S^{4k};\intg) \subset H^* (S^{4k};\rat)$.
Thus it is an isomorphism
\[ \ch: \K^0 (S^{4k})
   \stackrel{\cong}{\longrightarrow}
   H^* (S^{4k};\intg) = \intg \oplus \intg. \]
It restricts to an isomorphism
\[ \ch: \widetilde{\K}^0 (S^{4k})
   \stackrel{\cong}{\longrightarrow}
   \widetilde{H}^* (S^{4k};\intg) = \intg \]   
between reduced groups. 
The localization at odd primes is an isomorphism
\[ \ch [\smlhf]: \widetilde{\K}^0 (S^{4k})[\smlhf]
   \stackrel{\cong}{\longrightarrow}
   H^{4k} (S^{4k})[\smlhf] = \intg [\smlhf]. \]   
We turn next to complexification.
This is a morphism $c: \KO \to \K$ of spectra which induces a ring homomorphism
$c_*: \pi_* (\KO) \to \pi_* (\K)$ on homotopy rings
(\cite[p. 304]{switzer}).
In degree $4$,
$c_* = 2: \pi_4 (\KO)=\intg \to \intg=\pi_4 (\K)$,
while in degree $8$,
$c_*: \pi_8 (\KO)=\intg \to \intg=\pi_8 (\K)$
is an isomorphism.
The localized ring homomorphism
\[ c_* [\smlhf]: \pi_* (\KO)[\smlhf] = \intg [\smlhf][a^{\pm 1}]
\longrightarrow \intg [\smlhf][\beta^{\pm 1}]= \pi_* (\K)[\smlhf]
\]
sends $a\in \pi_4 (\KO)[\smlhf]$ to
$\beta^2 \in \pi_4 (\K)[\smlhf]$.
In particular, 
$c_* [\smlhf]: \pi_{4k} (\KO)[\smlhf] \to 
  \pi_{4k} (\K)[\smlhf]$
is an isomorphism, mapping $a^k \mapsto \beta^{2k}$.
The Chern character of $\beta^k$ is given by $v^k$, where
$v\in H^2 (S^2;\intg)$ is the canonical generator, i.e.
the first Chern class
$c_1 (H)$ of the canonical (hyperplane) line bundle $H$ over
$S^2 = \cplx P^1$. 
The Pontrjagin character $\ph = \ch \circ c$ localizes to 
$\ph [\smlhf]$ given by the composition
\[ \xymatrix{
\widetilde{\KO}^0 (S^{4k})[\smlhf] \ar[dr]_{\ph [\smlhf]} 
    \ar[r]^{c_* [\smlhf]}_\cong
& \widetilde{\K}^0 (S^{4k})[\smlhf] \ar[d]^{\ch [\smlhf]}_\cong \\
& H^{4k} (S^{4k};\intg)[\smlhf].
} \]     
Here we have used suspension isomorphisms to identify
\[ \xymatrix@C=10pt{  
\widetilde{\KO}^0 (S^{4k}) \ar[d]_{c_*} \ar@{=}[r]^\sim &
  \widetilde{\KO}^{-4k} (S^0) \ar@{=}[r] &
    \KO^{-4k} (\pt) \ar@{=}[r] &
       \pi_{4k} (\KO) \ar[d]^{c_*} \\
\widetilde{\K}^0 (S^{4k}) \ar@{=}[r]^\sim &
  \widetilde{\K}^{-4k} (S^0) \ar@{=}[r] &
    \K^{-4k} (\pt) \ar@{=}[r] &
       \pi_{4k} (\K).       
} \]  
Therefore,
\[ \ph [\smlhf] (a^k)
  = \ch [\smlhf] (c_* [\smlhf] (a^k))
  = \ch [\smlhf] (\beta^{2k}) = v^{2k}. \]
The generator $v^{2k} \in H^{4k} (S^{4k};\intg)$ agrees with
the Thom class $u_\intg (\alpha) \in 
H^{4k} (\real^{4k} \cup \{ \infty \};\intg)$ of the trivial rank
$4k$-bundle $\alpha$ over a point.
Let $\iota: H^* (S^{4k};\ism) \hookrightarrow
H^* (S^{4k};\rat)$ be the injection induced by $\ism \subset \rat$.
By Lemma \ref{lem.phdeltaalpha},
\begin{align*}
\iota \ph [\smlhf] (\Delta (\alpha))
&= L^{-1} (\alpha) \cup u_\rat (\alpha) 
   = 1\cup u_\rat (\alpha) = u_\rat (\alpha) \\
& =\iota (u_\intg (\alpha)) 
  = \iota (v^{2k}) = \iota \ph [\smlhf] (a^k).
\end{align*}     
Since $\iota \ph [\smlhf]$ is injective,
\[ \Delta (\alpha) = a^k. \]  
\end{proof}

\section{The Classical Construction of the Siegel-Sullivan Orientation}
\label{sec.classicalsiegelsullivan}

Using Goresky-MacPherson's intersection homology, Witt spaces have been introduced by 
P. Siegel in \cite{siegel}
as a geometric cycle theory representing $\KO$-homology 
at odd primes. Sources on intersection homology include
\cite{gmih1}, \cite{gmih2}, \cite{kirwanwoolf}, 
\cite{friedmanihbook}, \cite{maximinthombook}, \cite{borel}, \cite{banagltiss}.
\begin{defn}
A \emph{Witt space} is an oriented PL pseudomanifold such that
the links $L^{2k}$ of odd codimensional PL intrinsic strata have
vanishing lower middle-perversity degree $k$ rational intersection homology,
$IH^{\bar{m}}_k (L^{2k};\rat)=0$.
\end{defn}
For example, pure-dimensional complex algebraic varieties are Witt spaces,
since they are oriented pseudomanifolds and 
possess a Whitney stratification whose strata all have
even codimension. The vanishing condition on the intersection
homology of links $L^{2k}$ is equivalent to requiring the canonical morphism
from lower middle to upper middle perversity intersection chain sheaves
to be an isomorphism in the derived category of sheaf complexes.
Consequently, these middle perversity intersection chain sheaves are
Verdier self-dual, and this induces global Poincar\'e duality for
the middle perversity intersection homology groups of a compact Witt space.
In particular, compact Witt spaces $X$ have a well-defined bordism invariant 
signature $\sigma (X)$ and $L$-classes $L_* (X)\in H_* (X;\rat)$ which agree with
the Poincar\'e duals of Hirzebruch's tangential $L$-classes when $X$ is smooth.
The notion of \emph{Witt spaces with boundary} 
can be introduced as pairs
$(X,\partial X),$ where $X$ is a PL space and $\partial X$ a 
stratum preservingly collared PL subspace of $X$ such that
$X-\partial X$ and $\partial X$ are both compatibly oriented Witt spaces.
Let $\Omega^\Witt_* (-)$ denote the bordism theory based on Witt cycles.
Elements of $\Omega^\Witt_n (Y)$ are
Witt bordism classes of continuous maps $f: X^n \to Y$ defined on
$n$-dimensional closed Witt spaces $X$. The theory $\Omega^\Witt_* (-)$
is a homology theory, whose coefficients have been computed by Siegel.
They are nontrivial only in nonnegative degrees divisible by $4$, where
they are given by
$\syml_{4k} (\rat),$ $4k>0,$ and by $\intg$ in degree $0$.

Let $X$ be a closed Witt space of dimension $n$.
Drawing on Sullivan's methods as laid out in 
\cite{sullivanmit} and \cite{sullivansinginspaces},
Siegel constructs in \cite{siegel} a canonical orientation class
\[ \mu_X \in \KO_n (X) \otimes \ism.  \]
(In fact, the class lives in connective $\KO$-homology.)
We shall refer to $\mu_X$ as the 
\emph{Siegel-Sullivan orientation class} of $X$.
Let us briefly outline Siegel's construction, which rests on two fundamental
facts due to Sullivan:
First, there is an exact sequence
\[ 
0 \to \KO^i (Y,B)\otimes \ism \longrightarrow
  \KO^i (Y,B)^\wedge \oplus \KO^i (Y,B)\otimes \rat
  \longrightarrow \KO^i (Y,B)^\wedge \otimes \rat \to 0, 
\]  
where $\KO^i (Y,B)^\wedge$ denotes the profinite completion of $\KO^i (Y,B)$
with respect to groups of odd order.
Second, the natural transformation 
$\Delta_{\SO *}: \Omega^\SO_i (Y,B) \to (\KO [\smlhf])_i (Y,B)$
induces a Conner-Floyd type isomorphism
\[ \Omega^\SO_{i+4*} (Y,B) \otimes_{\Omega^\SO_* (\pt)} \ism
   \cong \KO_i (Y,B) \otimes \ism \]
of $\intg/4\intg$-periodic theories for compact PL pairs $(Y,B)$,
\cite{sullivanmit}, \cite[p. 85]{madsenmilgram}.
Together with universal coefficient considerations, these two facts 
imply that elements of 
$\KO^i (Y,B)\otimes \ism$ are pairs $(\sigma_0, \tau_0)$
of homomorphisms
$\sigma_0: \Omega^\SO_{i+4*} (Y,B) \otimes \rat \to \rat$
and 
$\tau_0: \Omega^\SO_{i+4*} (Y,B;\rat/\ism) \to \rat/\ism$  
such that the \emph{periodicity relations}
\begin{equation} \label{equ.sigmatauperiodicityrelations}
\sigma_0 ([f][M\to \pt]) = \sigma (M)\cdot \sigma_0 [f],~
\tau_0 ([f][M\to \pt]) = \sigma (M)\cdot \tau_0 [f]
\end{equation}
with respect to multiplication by a closed manifold $M$ hold and the diagram
\[ \xymatrix{
\Omega^\SO_{i+4*} (Y,B) \otimes \rat \ar[d] \ar[r]^>>>>>>>>{\sigma_0} 
   & \rat \ar[d] \\
\Omega^\SO_{i+4*} (Y,B;\rat/\ism) \ar[r]^>>>>>{\tau_0} & \rat/\ism
} \]
commutes.
To define $\mu_X$ for a closed Witt space $X^n$,
choose a PL embedding $X\subset \real^m$, $m$ large, of codimension $4k$.
Let $(N,\partial N)$ be a regular neighborhood of $X$.
We will describe an element in
$\KO^{4k} (N,\partial N)\otimes \ism,$ which corresponds under 
Alexander-Spanier-Whitehead
duality to $\mu_X \in \KO_n (X)\otimes \ism$.
Therefore, we need to specify homomorphisms
$(\sigma_X, \tau_X)$ satisfying the above periodicity relations and
the integrality condition, i.e. commutativity of
\[ \xymatrix{
\Omega^\SO_{4(k+*)} (N,\partial N) \otimes \rat 
          \ar[d] \ar[r]^>>>>>>>>{\sigma_X} & \rat \ar[d] \\
\Omega^\SO_{4(k+*)} (N,\partial N;\rat/\ism) 
   \ar[r]^>>>>>{\tau_X} & \rat/\ism.
} \]
The homomorphism $\sigma_X$ is 
\[ \sigma_X ([(M,\partial M) \stackrel{f}{\longrightarrow} (N,\partial N)]
  \otimes r)
    := \sigma (\widetilde{f}^{-1} (X)) \otimes r,~ r\in \rat, \]
where one uses the block-transversality results of
\cite{buonrs}, \cite{mccrory} to make $f$ transverse to $X$
in the PL manifold $N$. The preimage $\widetilde{f}^{-1} (X) \subset M$
under the transverse map $\widetilde{f}$ has the same local structure
as $X$ and thus is again a Witt space with a well-defined 
signature $\sigma (\widetilde{f}^{-1} (X)) \in \intg$.
The homomorphism $\tau_X$ is obtained by specifying a sequence 
of homomorphisms
\[ \tau_{X,k}: \Omega^\SO_* (N,\partial N; \intg/_k) 
   \longrightarrow \intg/_k,~ k \text{ odd}, \]
compatible with respect to divisibility,
which are defined in much the same way as $\sigma_X$, but
using oriented $\intg/_k$-manifolds to represent elements of
$\Omega^\SO_* (N,\partial N; \intg/_k)$. By Novikov additivity
for the signature of compact Witt spaces with boundary, the
transverse inverse image of $X$ in the
$\intg/_k$-manifold has a well-defined (and bordism invariant)
signature in $\intg/_k$, which defines $\tau_{X,k}$.
The periodicity and integrality conditions are satisfied
and thus an element $\mu_X$ is obtained.\\

The homomorphism
$c_*: \KO_n (X)\otimes \ism \to \KO_n (\pt) \otimes \ism$
induced by the constant map $c: X\to \pt$ sends 
$\mu_X$ to the signature of $X$.
Using the orientation class $\mu_X$, Siegel obtains a natural
transformation 
\[ \mu^\Witt: \Omega^\Witt_* (-) \longrightarrow \KO [\smlhf]_* (-) \]
of homology theories by setting
\[ \mu^\Witt ([X \stackrel{f}{\longrightarrow} Y]) 
   = f_* (\mu_X). \]
This transformation then reduces to the signature homomorphism
on coefficient groups. In terms of the transformation, the orientation
class can of course be recovered as
\[ \mu_X = \mu^\Witt ([\id_X]). \]
Siegel's transformation factors through the homomorphism induced
by the connective cover $\ko [\smlhf] \to \KO [\smlhf],$ since
$\Omega^\Witt_* (-)$ is connective.

\begin{thm} (Siegel.) \label{thm.siegelconnkoiswittbord}
The natural transformation
\[ \mu^\Witt [\smlhf]: \Omega^\Witt_* (-) \otimes \ism
   \longrightarrow \ko [\smlhf]_* (-) \otimes \ism \]
is an equivalence of homology theories.
\end{thm}

\begin{prop}(Siegel.) \label{prop.sullivsiegelorientonmanifolds}
If $X=M$ is a
smooth compact manifold, then $\mu_M$ agrees with the Sullivan
orientation $\Delta_\SO (M)$,
\[ \mu_M = \Delta_\SO (M). \]
\end{prop}
\begin{proof}
As pointed out by Siegel \cite[p. 1069]{siegel}, the statement
follows directly from the above construction, since it has been used by
Sullivan \cite{sullivanmit} in the case of a manifold to
construct his canonical $\KO [\smlhf]$-orientation.
(In fact, as pointed out in \cite{sullivansinginspaces} and by Siegel,
the construction applies in general to any bordism theory based
on a class $\mathcal{F}$ of PL spaces which is closed under taking
cartesian product with a PL manifold and intersecting transversely with
a closed manifold in Euclidean space, carries a bordism invariant
signature which satisfies Novikov additivity and is multiplicative
with respect to taking products with closed manifolds. Beyond 
PL manifolds and Witt spaces there exist much larger classes of 
singular spaces $\Fa$ that satisfy this, in particular the class
of stratified pseudomanifolds that admits Lagrangian structures along
strata of odd codimension, considered in \cite{banagl-mem}, 
\cite{banagl-lcl}.)
\end{proof}

\section{Ring Spectrum Level Construction of the Siegel-Sullivan Orientation}
\label{sec.ringspeclevsiegelsullivan}

Let $\MWITT$ be the Quinn spectrum associated to the ad-theory of Witt spaces,
representing Witt bordism, see Banagl-Laures-McClure \cite{blm}.
A weakly equivalent spectrum had been considered 
first by Curran in \cite{curran}. He verified that this spectrum is an
$\MSO$-module (\cite[Thm. 3.6, p. 117]{curran}).
The product of
two Witt spaces is again a Witt space. This implies essentially that
$\MWITT$ is a ring spectrum; for more details see \cite{blm}.
(There, we focused on the spectrum
$\operatorname{MIP}$ representing
bordism of integral intersection homology Poincar\'e spaces studied by
Goresky and Siegel in \cite{gorsie} and by Pardon in \cite{pardon}, 
but everything works in an analogous, indeed simpler, manner for $\rat$-Witt spaces.)
Every oriented PL manifold is a Witt space. Hence there is a map
\[ \phi_W: \MSPL \longrightarrow \MWITT, \]
which, using the methods of ad-theories and Quinn spectra
employed in \cite{blm}, can be constructed to be multiplicative.
In \cite{blm}, we constructed a map
\[ \tau: \MWITT \longrightarrow \syml (\rat). \]
(We even constructed an integral map
$\operatorname{MIP} \to \syml$.)
This map is multiplicative, i.e. a ring map, as shown
in \cite[Section 12]{blm}, and the diagram
\begin{equation} \label{equ.esigmaphifistauphiw}
\xymatrix{
\MSPL \ar[r]^{\sigma^*} \ar[d]_{\phi_W} & \syml (\intg) \ar[d] \\
\MWITT \ar[r]_\tau & \syml (\rat)
} \end{equation}
homotopy commutes, since it comes from a
commutative diagram of ad-theories under applying
the symmetric spectrum functor $\mathbf{M}$ of 
Laures and McClure \cite{lauresmcclure}. 
The localization morphism $\syml (\rat) \to \syml (\rat)[\smlhf]$
is a morphism of ring spectra. Thus the composition of $\tau$
with the localization morphism is a morphism of ring spectra
$\MWITT \to \syml (\rat)[\smlhf] = \syml (\intg)[\smlhf]$, 
which we shall also denote by $\tau$.
It was known to the experts early on that carrying out
Mishchenko's method \cite{mishchenko} with intersection chains rather than
ordinary chains would lead to an extension of the symmetric signature 
to pseudomanifolds with only even codimensional strata and, more
generally, to Witt spaces; see e.g.
\cite{weinberger}, \cite{cappshanstratifmaps},
\cite{banaglmsri}. In the context of their Witt package program
\cite{almpsigpack}, Albin, Leichtnam, Mazzeo and Piazza
applied this symmetric signature in defining a $C^*$-algebraic
Witt symmetric signature in $K_* (C^*_r \pi)$ which agrees
rationally with the index class of the signature operator
$\eth_{\operatorname{sign}}$.
In \cite{blm}, Laures, McClure and the author adopt the approach outlined in
\cite{banaglmsri} to construct the symmetric signature of
Witt and integral intersection Poincar\'e spaces:
The morphism $\tau: \MWITT \to \syml (\rat)$ of ring spectra
yields an $\syml (\rat)$-homology fundamental class for 
$n$-dimensional closed Witt spaces $X$ by setting
\[ [X]_\syml := \tau [\id_X] \in \syml (\rat)_n (X),  \]
$[\id_X] \in \Omega^\Witt_n (X)$. This fundamental class
yields the symmetric signature
\[ \sigma^* (X) = A[X]_\syml \in L^n (\rat [\pi_1 X]) \]
under the assembly map
\[ A: \syml (\rat)_n (X) \longrightarrow L^n (\rat [\pi_1 X]). \]
A detailed account of extending Mishchenko's approach to
intersection chains has been provided by Friedman and McClure
in \cite{friedmanmccluresymmsig}. By \cite[Thm. 10.12]{blm},
the above symmetric signature $\sigma^* (X)$ agrees with the
construction of Friedman-McClure.
According to \cite[Prop. 5.20]{friedmanmccluresymmsig},
the homomorphism
$L^n (\rat [\pi_1 X]) \to L^n (\rat)$ for $n=4k$ 
maps the symmetric signature $\sigma^* (X)$ to 
the Witt class $w(X)$ of the intersection form on
$IH^{\bar{m}}_{2k} (X;\rat)$. 
Under the localization homomorphism
$L^n (\rat) \to L^n (\rat)\otimes \ism,$
$w(X)$ maps to
the ordinary signature, 
$w(X)_{(\operatorname{odd})} = \sigma (X)\cdot x^k$, $n=4k$,
as in the manifold case of Remark \ref{rem.symmsigntoordsignformfds}.
The image of $[X]_\syml$ under the localization homomorphism
$\syml (\rat)_n (X) \to \syml (\rat)_n (X) \otimes \ism$
will again be denoted by $[X]_\syml$.

\begin{prop} \label{prop.tauagreeswithsiegelw}
In positive degrees, the composition
\[ \Omega^\Witt_* (\pt) \stackrel{\tau}{\longrightarrow} 
  \syml (\rat)_* (\pt) \stackrel{A}{=}
  L^* (\rat)  \]
agrees with the map $w: \Omega^\Witt_* (\pt) \to L^* (\rat)$
which sends the bordism class of a $4k$-dimensional Witt space $X$ to the 
Witt class $w(X)$ of its intersection form on
$IH^{\bar{m}}_{2k} (X;\rat)$ and is zero in degrees not divisible by $4$.
\end{prop}
\begin{proof}
The group $\Omega^\Witt_i (\pt)$ is zero in 
degrees that are not divisible by $4$.
Thus $A\circ \tau$ agrees trivially with $w$ in such degrees.
Let $[X] \in \Omega^\Witt_{4k} (\pt),$ $k>0,$ be any element.
Let $c:X\to \pt$ denote the constant map and consider the diagram
\begin{equation} \label{dia.constmapswitt} 
\xymatrix{
\Omega^\Witt_{4k} (X) \ar[r]^{\tau} \ar[d]_{c_*} &
 \syml (\rat)_{4k} (X) \ar[r]^A \ar[d]_{c_*} &
   L^{4k} (\rat [\pi_1 X]) \ar[d]_{c_*} \\
\Omega^\Witt_{4k} (\pt) \ar[r]_{\tau} &
 \syml (\rat)_{4k} (\pt) \ar@{=}[r]_A &
   L^{4k} (\rat).   
} \end{equation}
Since both $\tau$ and the assembly map are natural,
the diagram commutes.
The right hand vertical map sends the symmetric signature
$\sigma^* (X) \in L^{4k} (\rat [\pi_1 X])$ 
to $w(X)$. Therefore,
\begin{align*}
A\tau [X]
&= A\tau c_* [\id_X] = c_* A \tau [\id_X] \\
&= c_* A [X]_\syml = c_* \sigma^* (X) = w(X).
\end{align*}
\end{proof}

The morphism $\tau: \MWITT \to \syml (\rat)$ is not an equivalence.
One reason is that $\MWITT$ is connective whereas $\syml (\rat)$ is periodic.
Let $t_{\geq m} E \to E$ be the
$(m-1)$-connective cover of a spectrum $E$.
A morphism $\phi: E\to F$ of spectra 
lifts, uniquely up to homotopy, to a morphism
$t_{\geq m} \phi: t_{\geq m} E \to t_{\geq m} F$.
The lift $t_{\geq 0} \tau$ to the connective cover 
$t_{\geq 0} \syml (\rat)$ is still no equivalence,
as $\pi_0 \MWITT = \Omega^\Witt_0 (\pt) = \intg$, while
$\pi_0 \syml (\rat) = L^0 (\rat)$ contains an infinitely generated
amount of $2$-primary torsion.
Degree zero is, however, the only offending nonnegative degree:
\begin{thm}
The lift $t_{\geq 1} \tau: 
t_{\geq 1} \MWITT \to t_{\geq 1} \syml (\rat)$
is a weak equivalence.
\end{thm}
\begin{proof}
By Proposition \ref{prop.tauagreeswithsiegelw},
the diagram
\[ \xymatrix{
\pi_i (t_{\geq 1} \MWITT) \ar[r]^{t_{\geq 1} \tau} \ar[d]_\cong &
  \pi_i (t_{\geq 1} \syml (\rat)) \ar[d]^\cong \\ 
\pi_i (\MWITT) \ar[r]^\tau \ar[rd]_w &
  \pi_i (\syml (\rat)) \ar@{=}[d]^A \\   
 & L^i (\rat)
} \]
commutes for $i\geq 1$.
Siegel's Witt bordism calculation \cite[Prop 1.1, p. 1098]{siegel}
asserts that $w$ is an isomorphism.
\end{proof}

The central construction of the present paper is the following lift of the
Siegel-Sullivan orientation introduced in \cite{siegel} to the ring-spectrum
level.
\begin{defn}
The \emph{ring-spectrum level Siegel-Sullivan orientation}
\[ \Delta: \MWITT \longrightarrow \KO [\smlhf] \]
is the morphism of ring spectra given by the composition
\[ \MWITT \stackrel{\tau}{\longrightarrow} 
  \syml (\rat)[\smlhf] = \syml (\real)[\smlhf] 
   \stackrel{\kappa^{-1}}{\simeq} \KO [\smlhf]  \]
of the Witt-orientation introduced by
Laures, McClure and the author
with a ring equivalence $\kappa^{-1}$
inverse to the ring equivalence $\kappa$ of
Proposition \ref{prop.einftyequivkol}.
\end{defn}
Diagram (\ref{equ.esigmaphifistauphiw}) then embeds into the
homotopy commutative diagram
\[
\xymatrix{
\MSPL \ar@/^2pc/[rr]^\Delta
\ar[r]^{\sigma^*} \ar[d]_{\phi_W} & \syml (\intg)[\smlhf] \ar[d]^\simeq 
  & \KO [\smlhf] \ar[l]^\kappa_\simeq \ar@{=}[d] \\
\MWITT \ar@/_2pc/[rr]_\Delta
\ar[r]_\tau & \syml (\rat)[\smlhf]
  & \KO [\smlhf]. \ar[l]^\simeq
} \]
Thus the ring-spectrum level Siegel-Sullivan orientation
$\Delta: \MWITT \to \KO [\smlhf]$ restricts under 
$\phi_W: \MSPL \to \MWITT$ to the
Sullivan orientation 
\[ \MSPL \stackrel{\Delta}{\longrightarrow} \KO [\smlhf] \]
of Definition \ref{def.sullorientringmap}.
In order to describe the induced map 
$\Delta_*: \MWITT_* \to \KO [\smlhf]_*$ on coefficients,
we shall employ the symmetric signature of Witt spaces.
We observe that (\ref{equ.sullorientonplcoeffs}) extends from the 
manifold case to Witt spaces:

\begin{prop} \label{prop.speclevelsiegelsullonhtpygrps}
The ring-spectrum level Siegel-Sullivan orientation $\Delta$
induces on homotopy groups the homomorphism
\[ \Delta_*: \Omega^\Witt_{4k} = \MWITT_{4k} \longrightarrow
   \KO [\smlhf]_{4k} = \intg [\smlhf] \langle a^k \rangle \]
given by
\begin{equation} \label{equ.siegelsullorientoncoeffs} 
\Delta_* [X^{4k}] = \sigma (X)\cdot a^k,
\end{equation}
where $\sigma (X)$ is the signature of the intersection form
on the intersection homology groups $IH_{2k} (X;\rat)$ of $X$.
\end{prop}
\begin{proof}
Let $[X] \in \Omega^\Witt_{4k}$ be any element.
The constant map $c:X\to \pt$ induces a commutative diagram
\[ \xymatrix{
\Omega^\Witt_n (X) \ar[r]^>>>>>{\tau} \ar[d]_{c_*} &
 \syml (\rat)_n (X) \otimes \ism \ar[r]^{A [\frac{1}{2}]} \ar[d]_{c_*} &
   L^n (\rat [\pi_1 X]) \otimes \ism \ar[d]_{c_*} \\
\Omega^\Witt_n (\pt) \ar[r]_>>>>>{\tau} &
 \syml (\rat)_n (\pt) \otimes \ism \ar@{=}[r]_{A [\frac{1}{2}]} &
   L^n (\rat) \otimes \ism.   
} \]
This is quite similar to Diagram (\ref{dia.constmapswitt}), except that 
here we map into L-theory away from $2$.
The claim is established by the calculation
\begin{align*}
\kappa_* \Delta_* [X]
&= \tau_* [X] = \tau_* c_* [\id_X] 
  = c_* \tau_* [\id_X] = c_* [X]_\syml \\
&= c_* (A[\smlhf]) [X]_\syml = c_* \sigma^* (X)_{(\operatorname{odd})} 
   = w(X)_{(\operatorname{odd})}
   = \sigma (X) \cdot x^k,
\end{align*}
so that
\[ \Delta_* [X] = \sigma (X) \cdot \kappa^{-1}_* (x^k) =
   \sigma (X) \cdot a^k.  \]
\end{proof}

The ring-spectrum level Siegel-Sullivan orientation allows in particular for
the following description of the
Siegel-Sullivan orientation class of a Witt space:
\begin{defn}
The \emph{Siegel-Sullivan orientation class} of a compact $n$-dimensional Witt space
$(X,\partial X)$ is given by the image
\[ \Delta (X) := \Delta_* [\id_X] \in \KO_n (X,\partial X)\otimes \intg [\smlhf]  \]
of the Witt bordism class of the identity on $X$ under the
homomorphism
$\Delta_*: \Omega^\Witt_n (X,\partial X) \to (\KO [\smlhf])_n (X,\partial X)$
induced by the ring-spectrum level Siegel-Sullivan orientation $\Delta$.
\end{defn} 
We will see in Proposition \ref{prop.deltaagreeswithmuwitt} below that 
this terminology is justified, i.e. that $\Delta (X) = \mu_X$.
By our construction,
\begin{equation} \label{equ.kappadeltaissymlor}
\kappa_* \Delta (X) = \kappa_* \Delta_* [\id_X] = \tau [\id_X]
   = [X]_\syml \in (\syml \rat [\smlhf])_n (X,\partial X). 
\end{equation}
Since $\Delta: \MWITT \to \KO [\smlhf]$ restricts to the Sullivan orientation
$\Delta: \MSPL \to \KO [\smlhf]$, the Siegel-Sullivan orientation
$\Delta (X)$ agrees with the Sullivan orientation 
$\Delta_\SPL (X)$ when $X$ is a PL-manifold.
As a final point of business in setting up the
spectrum level Siegel-Sullivan orientation, we shall verify that the
induced transformation of homology theories agrees with the classical
construction as given by Siegel and outlined in 
Section \ref{sec.classicalsiegelsullivan}.
The argument is based on a result (\cite[Prop. 2, p. 597]{banaglcappshan})
of Cappell, Shaneson and the author which shows that at odd primes,
Witt bordism classes are representable by smooth oriented bordism classes.
Let us review this result briefly.

For an integer $j$, let $\bar{j}$ denote its residue class in $\intg/_4$.
Using $4$-fold periodicity, we may view $\KO \smlhf$-homology as
$\intg/_4$-graded. On the groups 
$C_{\bar{j}} (X,Y) := \bigoplus_{k\in \intg} \Omega^\SO_{j+4k} (X,Y) \otimes \ism$, 
define an equivalence relation by
\[
 [M^{j+4k} \times N^{4i} \stackrel{\proj}{\longrightarrow} M 
      \stackrel{f}{\longrightarrow} X]
 \sim
  \sigma (N) \cdot [M^{j+4k} \stackrel{f}{\longrightarrow} X],
\]
where $\sigma (N)$ is the signature of the manifold $N$.
(See also \cite[p. 193]{krecklueck}.)
The periodicity relations (\ref{equ.sigmatauperiodicityrelations}) imply that
Sullivan's orientation $\Delta_{\SO *}$ induces a well-defined homomorphism
\begin{equation} \label{equ.sullivanconnfloydiso} 
\Delta_{\SO *}: Q_{\bar{j}} (X,Y) \longrightarrow 
  (\KO \smlhf)_{\bar{j}} (X,Y) 
\end{equation}  
on the quotient $Q_{\bar{j}} (X,Y) := C_{\bar{j}} (X,Y)/\sim$.
This is a natural transformation of functors, and
Sullivan proves that it is an isomorphism for compact PL pairs $(X,Y)$
(\cite{sullivanmit}, \cite[4.15, p. 85]{madsenmilgram}).
This shows in particular that $(X,Y)\mapsto Q_{\bar{j}} (X,Y)$
is a ($\intg/_4$-graded) homology theory on compact PL pairs.
Let $Z$ denote the ring $Z=\ism$.
The Laurent polynomial ring $Z[t,t^{-1}]$
is a $\intg$-graded ring with $\deg (t)=4$.
There is a canonical subring inclusion
$Z[t] \subset Z[t,t^{-1}],~ t \mapsto t.$
Via this inclusion, $Z[t,t^{-1}]$ becomes a $Z[t]$-module and
Panov observes in \cite{panov} that this module is flat.
As the connective spectrum $\ko \smlhf$ is a ring spectrum,
$(\ko \smlhf)_* (X)$ is in particular a right $(\ko \smlhf)_* (\pt)=Z[a]$-module,
where $a\in (\ko \smlhf)_4 (\pt) = Z$ is the generator
which complexifies to the square of the complex Bott element,
as in Section \ref{sec.multidentkol}.
For periodic $\KO$, we have
$(\KO \smlhf)_* (\pt) = Z[a,a^{-1}],$
and the canonical map
$(\ko \smlhf)_* (\pt) \to
    (\KO \smlhf)_* (\pt)$
is given by the inclusion
$Z[a] \subset Z[a,a^{-1}],~ a\mapsto a.$
Using the isomorphism $Z[t] \cong Z[a],~ t \mapsto a,$
$(\ko \smlhf)_* (X)$ and $(\KO \smlhf)_* (X)$ become right $Z[t]$-modules.
We may therefore form the tensor product of $Z[t]$-modules
\[ (\ko \smlhf)_* (X) \otimes_{Z[t]} Z[t,t^{-1}]. \]
Since the functor $- \otimes_{Z[t]} Z[t,t^{-1}]$ is exact
by Panov's observation, the functor
\[ (X,Y) \mapsto (\ko \smlhf)_* (X,Y) \otimes_{Z[t]} Z[t,t^{-1}] \]
is exact and thus a homology theory (on CW pairs $(X,Y)$), 
which is $\intg$-graded by
$\deg (x \otimes_{Z[t]} rt^k) = n+4k,$
$x\in (\ko \smlhf)_n (X,Y),$ $r\in Z$.
Now consider $\KO_* (-)$ as a $\intg$-graded
(not $\intg/_4$-graded) theory.
The connective cover
$(\ko \smlhf)_* (-) \to (\KO \smlhf)_* (-)$ is $Z[t]$-linear and hence
induces a natural transformation 
\[ \Phi: (\ko \smlhf)_* (-) \otimes_{Z[t]} Z[t,t^{-1}]
  \longrightarrow
  (\KO \smlhf)_* (-) \]
of $\intg$-graded homology theories. 
On a point, $\Phi$ is the identity map. Thus $\Phi$ is a
natural equivalence of homology theories.
This shows how to reconstruct periodic $\KO \smlhf$-homology
from connective $\ko \smlhf$-homology and the action of $Z[t]$ on it.
Similarly, we can turn Witt bordism, which is a connective theory, into a 
periodic version:
Denote Witt bordism away from $2$ by
\[ W_* (X) := \Omega^\Witt_* (X)\otimes \ism. \]
This is a $\intg$-graded homology theory with coefficients
\[ W_* (\pt) = Z[c],~ c:= [\cplx P^2]\otimes 1 
       \in W_4 (\pt). \]
(After inverting $2$, only the signature survives as an
invariant; the $2$- and $4$-torsion is killed.)
The $\intg$-graded abelian group $W_* (X)$ is a right module over the ring $W_* (\pt)$
as usual.
The ring isomorphism
$Z[t] \to W_* (\pt)$
induced by
$t \mapsto c\otimes 1 \in W_4 (\pt)$
makes $W_* (X)$ into a right $Z[t]$-module.
We may thus form the tensor product
\[ \overline{W}_* (X) := W_* (X) \otimes_{Z[t]} Z[t,t^{-1}], \]
which is $\intg$-graded by
$\deg (x \otimes_{Z[t]} rt^k) = n+4k,$
$x\in W_n (X),$ $r\in Z.$
Panov's observation shows that $\overline{W}_* (-)$ is a
homology theory.
It is naturally a right $Z[t,t^{-1}]$-module and right multiplication
with $t$ is an isomorphism with inverse given by right multiplication
with $t^{-1}$. This shows that $\overline{W}_* (-)$ is $4$-periodic
so that we may call it \emph{periodic Witt-bordism at odd primes}. 
The inclusion $Z[t] \subset Z[t,t^{-1}]$ induces a natural map
\[ i_*: W_* (X) = W_* (X) \otimes_{Z[t]} Z[t] \longrightarrow
   \overline{W}_* (X). \]
Again, the periodicity relations (\ref{equ.sigmatauperiodicityrelations}) imply
that
\begin{equation} \label{equ.muwittperiodicity}
\mu^\Witt ([f:V^{j-4k} \to X]\cdot [M^{4k}]) = 
      (\mu^\Witt [f]) \cdot \sigma (M) a^k 
      \in (\ko \smlhf)_j (X),
\end{equation}  
$[f] \in \Omega^\Witt_{j-4k} (X),$ $[M]\in \Omega^\SO_{4k} (\pt)$,
which shows that
$\mu^\Witt$ is a homomorphism of $Z[t]$-modules, as is $\Delta_{\SO *}$
in the manifold case.
Tensoring over $Z[t]$ with $Z[t,t^{-1}]$, we get a natural isomorphism
\[ \mu^\Witt \otimes_{Z[t]} \id: 
   \overline{W}_* (X) = W_* (X) \otimes_{Z[t]} Z[t,t^{-1}]
   \stackrel{\cong}{\longrightarrow}
   (\ko \smlhf)_* (X) \otimes_{Z[t]} Z[t,t^{-1}]
    \]
of $\intg$-graded homology theories by
Siegel's Theorem \ref{thm.siegelconnkoiswittbord}.
For any $j\in \intg$, a well-defined map
\[  \omega: Q_{\bar{j}} (X) \otimes Z
   \longrightarrow (W_* (X) \otimes_{Z[t]} Z[t, t^{-1}])_j    \]
is given by setting
\[   \omega ([g: M^{j-4k} \to X] \otimes_{\intg} r)
   := [g] \otimes_{Z[t]} rt^k  
     \in W_{j-4k} (X) \otimes_{Z[t]} Z \langle t^k \rangle,~ 
     k \in \intg,~ r \in Z,   \]
where one views the closed oriented smooth manifold $M$
as a Witt space via its canonical PL structure.
On compact PL spaces $X$, the diagram
\[ \xymatrix@C=80pt{
Q_{\bar{j}} (X)\otimes Z \ar[r]^\cong_{\Delta_{\SO *}} \ar[d]_\omega
  & (\KO \smlhf)_{\bar{j}} (X) \\
  \overline{W}_j (X)
   \ar[r]^\cong_{\mu^\Witt \otimes_{Z[t]} \id} &
((\ko \smlhf)_* (X) \otimes_{Z[t]} Z[t,t^{-1}])_j     
  \ar[u]_\cong^\Phi
} \]
commutes for every $j\in \intg$ 
by Proposition \ref{prop.sullivsiegelorientonmanifolds}. 
In particular, $\omega$ is an isomorphism, from
which representability of Witt bordism classes by smooth manifolds,
away from $2$, can be deduced. The following consequence will
be used in the proof of Proposition \ref{prop.deltaagreeswithmuwitt}:

\begin{prop} \label{prop.conntoperiodicfactors}
1. Given a $Z[t]$-linear map
$\alpha_*: W_* (X) \to (\KO \smlhf)_* (X),$
there exists a unique extension of $\alpha_*$ to a homomorphism 
\[ \overline{\alpha}_*: \overline{W}_* (X) \longrightarrow
    (\KO \smlhf)_* (X)  \]
of $Z[t,t^{-1}]$-modules. If $\alpha_*$ is a natural transformation of
homology theories, then so is $\overline{\alpha}_*$.

2. Let $\alpha_*, \beta_*: W_* (X) \to (\KO \smlhf)_* (X)$
be $Z[t]$-linear natural transformations of homology theories
on compact PL $X$. If 
$\alpha_* ([g:M \to X]\otimes 1) = \beta_* ([g]\otimes 1)$ for
every $g$ on smooth manifolds $M$, then 
$\alpha_* = \beta_*$ on $W_* (X)$, and
$\overline{\alpha}_* = \overline{\beta}_*$
on $\overline{W}_* (X)$ for their periodic versions.
\end{prop}
\begin{proof}
We prove statement 1:
Since $\alpha_*$ is $Z[t]$-linear, it induces
a map
\[ \overline{\alpha}_*:
   (W_* (X) \otimes_{Z[t]} Z[t,t^{-1}])_j \longrightarrow
    (\KO \smlhf)_{\bar{j}} (X),~ j\in \intg, \]
by setting, for $p\in Z[t,t^{-1}],$    
\[ \overline{\alpha}_* ([f] \otimes_{Z[t]} p) 
     := (\alpha_* [f])\cdot p, \] 
where on the right hand side, we interpret $p$ as an element
of $Z[a,a^{-1}]$ by substituting $t \mapsto a$.     
Then $\overline{\alpha}_*$ is $Z[t,t^{-1}]$-linear, and
the diagram 
\begin{equation} \label{equ.overalphaextendsalpha}
\xymatrix@C=50pt@R=40pt{
W_* (X) \ar[rd]^{\alpha_*} \ar[d]_{i_*} 
   &  \\
\overline{W}_* (X) \ar[r]_{\overline{\alpha}_*} & (\KO \smlhf)_* (X)
} 
\end{equation}
commutes. 

We turn to the proof of uniqueness.
Suppose that $\beta_*: \overline{W}_* (X) \to (\KO \smlhf)_* (X)$
is any $Z[t,t^{-1}]$-linear extension of $\alpha_*$,
i.e. $\beta_* \circ i_* = \alpha_*$. Then
\begin{align*}
\beta_* ([f] \otimes_{Z[t]} p)
&= \beta_* ([f] \otimes_{Z[t]} (1 \cdot p)) 
 = \beta_* (([f] \otimes_{Z[t]} 1) \cdot p) \\
&= (\beta_* ([f] \otimes_{Z[t]} 1)) \cdot p 
  = (\beta_* i_* [f]) \cdot p \\
&= \alpha_* [f] \cdot p 
  = \overline{\alpha}_* ([f] \otimes_{Z[t]} p).
\end{align*}
Hence $\overline{\alpha}_*$ is unique.
If $\alpha_*$ is natural in $X$ and commutes with
suspension isomorphisms, then $\overline{\alpha}_*$ inherits these
properties.

We prove statement 2:
Since $\alpha_*$ and $\beta_*$ are $Z[t]$-linear, they induce
uniquely $Z[t,t^{-1}]$-linear transformations
\[ \overline{\alpha}_*, \overline{\beta}_*:
   (W_* (X) \otimes_{Z[t]} Z[t,t^{-1}])_j \longrightarrow
    (\KO \smlhf)_{\bar{j}} (X),~ j\in \intg, \]
as explained in statement 1.       
Given an element
\[ [f: V^{j-4k} \to X] \otimes_{Z[t]} rt^k 
   \in (W_* (X) \otimes_{Z[t]} Z[t,t^{-1}])_j
     = \overline{W}_j (X) \]
$k\in \intg,$ $r\in Z$,     
there exists a (unique) element $q\in Q_{\bar{j}} (X)\otimes Z$
with $\omega (q)= [f] \otimes_{Z[t]} rt^k$,
as $\omega$ is an isomorphism.
Such an element is represented in the quotient $Q_{\bar{j}} (X)\otimes Z$ 
by an element of the form
\[
q = \sum_{i=1}^m [g_i: M_i^{j-4k_i} \to X]\otimes r_i,~
  [g_i] \in \Omega^\SO_{j-4k_i} (X),~ r_i \in Z,~ k_i \in \intg. \]
By the definition of $\omega,$
$\omega ([g_i] \otimes_{\intg} r_i) = [g_i] \otimes_{Z[t]} r_i t^{k_i},$
so that
\[ [f] \otimes_{Z[t]} rt^k = \sum_{i=1}^m 
   [g_i] \otimes_{Z[t]} r_i t^{k_i}  \]
and consequently,   
\begin{align*} 
\overline{\alpha}_* ([f] \otimes_{Z[t]} rt^k) 
&= \sum_{i=1}^m 
   \overline{\alpha}_* ([g_i] \otimes_{Z[t]} r_i t^{k_i}) 
   = \sum_{i=1}^m 
     (\alpha_* [g_i]) \cdot r_i a^{k_i} \\
&= \sum_{i=1}^m 
   (\beta_* [g_i]) \cdot r_i a^{k_i} 
   = \sum_{i=1}^m 
   \overline{\beta}_* ([g_i] \otimes_{Z[t]} r_i t^{k_i}) \\       
&= \overline{\beta}_* ([f] \otimes_{Z[t]} rt^k).    
\end{align*}   
This proves that the periodic versions agree on $\overline{W}_* (X)$, 
$\overline{\alpha}_* = \overline{\beta}_*$.
Using the commutativity of (\ref{equ.overalphaextendsalpha})
we deduce
$\alpha_* = \overline{\alpha}_* \circ i_*
    = \overline{\beta}_* \circ i_* = \beta_*.$
\end{proof}

\begin{prop} \label{prop.deltaagreeswithmuwitt}
The natural transformation of homology theories induced by
$\Delta: \MWITT \to \KO [\smlhf]$ agrees with the transformation
$\mu^\Witt,$
\[ \Delta_* = \mu^\Witt: \Omega^\Witt_* (Y,B) \longrightarrow
             \KO_* (Y,B)\otimes \ism \]
on compact PL pairs $(Y,B)$. In particular, $\Delta (X) = \mu_X$ 
for a compact Witt space $(X,\partial X)$.             
\end{prop}
\begin{proof}
The ring-spectrum level Siegel-Sullivan orientation
$\Delta: \MWITT \to \KO [\smlhf]$ restricts to the
orientation 
$\Delta: \MSPL \to \KO [\smlhf]$ 
of Definition \ref{def.sullorientringmap}, and then further to
$\Delta|:\MSO \to \MSPL \stackrel{\Delta}{\longrightarrow}
\KO [\smlhf]$.
By Proposition \ref{prop.kappadeltasplisranickior},
and since $\Delta_\SPL$ extends $\Delta_\SO$,
the induced map
\[ \Delta|_*: \Omega^\SO_* (Y,B) \longrightarrow 
  \KO_* (Y,B) \otimes \ism \]
agrees with 
\[ \Delta_{\SO *}: \Omega^\SO_* (Y,B) \longrightarrow 
  \KO_* (Y,B) \otimes \ism \]
on every compact PL pair $(Y,B)$.
Since 
$\Delta_*: \Omega^\Witt_* (Y,B) \to \KO_* (Y,B)\otimes \ism$ 
is odd prime local, it factors uniquely through
the odd-primary localization $W_* (-)$ of $\Omega^\Witt_* (-)$.
The resulting homomorphism
\begin{equation} \label{equ.deltastaronwstar}
\Delta_*: W_* (Y,B) \longrightarrow (\KO \smlhf)_* (Y,B) 
\end{equation}
is a natural transformation of homology theories.
On a point, it is given by
\[ \Delta_* (c^k) = \Delta_{\SO *} (c^k) = 
   \Delta_{\SO *} [\cplx P^2]^k = \sigma (\cplx P^2)^k \cdot a^k =a^k,  \]
using the signature relation (\ref{equ.sullorientoncoeffs}).
Since $\Delta: \MWITT \to \KO [\smlhf]$ is a morphism of ring spectra,
this shows that (\ref{equ.deltastaronwstar}) is $Z[t]$-linear.
Siegel's transformation
\[ \mu^\Witt: W_* (Y,B) \longrightarrow (\KO \smlhf)_* (Y,B) \]
is a natural transformation of homology theories as well, and it is
$Z[t]$-linear by the periodicity relation (\ref{equ.muwittperiodicity}).
Thus, in order to prove equality of the two transformations on $W_* (Y,B)$,
it remains by Proposition \ref{prop.conntoperiodicfactors} to show that
they agree on compact smooth manifolds mapping into $(Y,B)$.
Let $g:(M,\partial M) \to (Y,B)$ be a smooth manifold in $(Y,B)$.
By Proposition \ref{prop.sullivsiegelorientonmanifolds},
$\mu_M = \Delta_\SO (M)$ and hence
\[ \Delta_* ([g]\otimes 1) 
  = \Delta_{\SO *} [g] 
  = g_* \Delta_\SO (M)
  = g_* \mu_M
= \mu^\Witt ([g]\otimes 1), \]
as was to be shown.
\end{proof}

\section{Multiplicativity of Orientation Classes}
\label{sec.siegelsullivanfullmultiplicativity}

The classical Sullivan orientation $\Delta_\SO (M \times N)$ 
of a product $M\times N$ of closed smooth manifolds is known to satisfy
the multiplicativity property
\begin{equation} \label{equ.deltasomfdtimesmfd}
\Delta_\SO (M\times N) = \Delta_\SO (M)\times \Delta_\SO (N) 
\end{equation}
with respect to the external product on $\KO [\smlhf]$-homology.
One way to see this
is to argue L-theoretically using Ranicki's
multiplicative morphism $\sigma^*: \MSPL \to \syml (\intg)$:
A ring morphism $\phi: E\to F$ between ring spectra
preserves external products, i.e. the diagram
\begin{equation} \label{dia.ringmorpreservesextprod}
\xymatrix{
E_* (X) \otimes E_* (Y) \ar[r]^{\phi \otimes \phi}
\ar[d]_\times & F_* (X) \otimes F_* (Y) \ar[d]^\times \\
E_* (X\times Y) \ar[r]_\phi & F_* (X\times Y)
} 
\end{equation}
commutes. Applying this to $\sigma^*$, one obtains the multiplicativity
of the $\syml$-homology orientation of PL manifolds,
\begin{align*}
[M]_\syml \times [N]_\syml
&= \sigma^* [\id_M] \times \sigma^* [\id_N] 
 = \sigma^* ([\id_M] \times [\id_N]) \\
&= \sigma^* [\id_{M \times N}] 
 = [M\times N]_\syml.
\end{align*}
(Via the ring morphism $\MSTOP \to \syml (\intg)$, 
this works just as well for topological manifolds.)
Using the ring equivalence
$\kappa^{-1}: \syml (\intg)[\smlhf] \simeq \KO [\smlhf],$
it follows that
$\Delta (M\times N) = \Delta (M)\times \Delta (N)$
for $\Delta: \MSPL \to \KO[\smlhf]$ as in 
Definition \ref{def.sullorientringmap}.
Equation (\ref{equ.deltasomfdtimesmfd}) is then a consequence of 
Proposition \ref{prop.kappadeltasplisranickior}.
Alternatively, and more directly, one may also deduce
(\ref{equ.deltasomfdtimesmfd}) 
from the multiplicativity of Sullivan's universal elements
\[ \Delta_{4n} \in \KO^{4n} (\MSO_{4n})\otimes \intg [\smlhf]. \]
Indeed, he shows (\cite[p. 202, g)]{sullivanmit}) that
the canonical homomorphism, induced by the classifying map,
\[ \widetilde{\KO \smlhf}^{4(q+r)} (\MSO_{4(q+r)})
   \longrightarrow
    \widetilde{\KO \smlhf}^{4(q+r)}  (\MSO_{4q} \Smash \MSO_{4r})
\]
sends
$\Delta_{4(q+r)} \mapsto \Delta_{4q} \times \Delta_{4r}.$  
This, together with naturality, implies the relation
\[ \Delta_\SO (\xi \times \eta)
    = \Delta_\SO (\xi) \times \Delta_\SO (\eta)
    \in \widetilde{\KO \smlhf}^* (\Th (\xi \times \eta))
      = \widetilde{\KO \smlhf}^* (\Th \xi \Smash \Th \eta) \]
for the Sullivan Thom class
$\Delta_\SO (\xi) \in \widetilde{\KO \smlhf} (\Th \xi)$
of an oriented vector bundle $\xi$ over a finite complex;
see also Madsen-Milgram \cite[p. 116]{madsenmilgram}.
Finally, one applies this to the stable normal bundles
$\xi = \nu_M,$ $\eta = \nu_N$ and uses Alexander duality.

\begin{remark}
As Rosenberg and Weinberger
point out in \cite[p. 51, Lemma 6]{rosenbergweinberger}, 
it is \emph{not} true that the class $[D_M]$ of the signature operator $D_M$
in Kasparov K-homology is multiplicative under external Kasparov product. 
If one of the dimensions of these classes is even, then the class
is multiplicative, but if both dimensions are odd, then
$[D_{M\times N}] = 2 [D_M] \times [D_N]$.
See also \cite[Theorem 8.5]{landnikolausschlichting}.
\end{remark}

A first immediate application of our approach, then, is a 
proof of full cartesian multiplicativity of the Siegel-Sullivan
orientation class
$\mu_X = \Delta (X) \in \KO_n (X)\otimes \ism$
for Witt spaces $X$, generalizing the manifold multiplicativity.
(This was not established in \cite{siegel}.)
\begin{thm} \label{thm.siegelsullivcartesianmult}
Let $X$ and $Y$ be closed Witt spaces. Then the Siegel-Sullivan
orientation of the Witt space $X\times Y$ is given by
\[ \Delta (X\times Y) = \Delta (X) \times \Delta (Y), \]
using the external product
$(\KO \smlhf)_m (X) \otimes (\KO \smlhf)_n (Y) \to
(\KO \smlhf)_{m+n} (X\times Y)$.
\end{thm}
\begin{proof}
Applying diagram (\ref{dia.ringmorpreservesextprod}) 
to the ring
morphism $\Delta: \MWITT \to \KO [\smlhf]$, we obtain
for the orientation classes
\begin{align*}
\Delta (X\times Y)
&= \Delta_* [\id_{X\times Y}]
    = \Delta_* ([\id_X] \times [\id_Y]) \\
&= \Delta_* [\id_X] \times \Delta_* [\id_Y]
   = \Delta (X) \times \Delta (Y).    
\end{align*}
\end{proof}

\section{Bundle Transfer of the Siegel-Sullivan Orientation}
\label{sec.bundletransfersiegelsullivan}

The equivalence of $\mathbb{E}_\infty$-ring spectra
$\kappa: \KO [\smlhf] \simeq \syml (\real) [\smlhf]$ constructed in
Proposition \ref{prop.einftyequivkol},
together with $\syml$-theoretic results of \cite{banaglbundletransfer},
allows for a treatment of bundle transfers of
Siegel-Sullivan classes. We begin with recollections on
homological block bundle transfer homomorphisms and use $\kappa$ to relate
the transfers on $\KO$- and $\syml$-homology.

Let $F$ be a closed oriented PL manifold 
of dimension $d$ and let $K$ be a finite ball complex with associated
polyhedron $B=|K|$. Let $b$ denote the dimension of $B$.
Let $\xi$ be an oriented PL $F$-block bundle over $K$
with total space $X=E(\xi)$, $\dim X = b+d$. The theory of
$F$-block bundles has been developed by Casson in \cite{casson}.
PL-locally trivial PL fiber bundles $X\to B$ with pointwise fiber $F$
are a special case.

Let $E$ be a ring spectrum equipped with a ring map $\MSPL \to E$.
Then the block bundle $\xi$ has an associated
block transfer homomorphism
\begin{equation} \label{equ.etransferofp}
\xi^!: E_n (B) \longrightarrow E_{n+d} (X).
\end{equation}
In \cite{banaglbundletransfer}, we described $\xi^!$
as a composition
\[ E_n (B) \stackrel{\sigma}{\cong}
  \widetilde{E}_{n+s} (S^s B^+)
   \stackrel{T(\xi)_*}{\longrightarrow} 
     \widetilde{E}_{n+s} (\Th (\nu))
      \stackrel{\Phi}{\longrightarrow}
 E_{n+s-(s-d)} (X). \]
Here, $\sigma$ is the suspension isomorphism and
$T(\xi): S^s B^+ \to \Th (\nu)$ the Umkehr map
(i.e. Pontrjagin-Thom collapse)
associated to a $\xi$-block preserving PL embedding
$X \hookrightarrow B \times \real^s$ for large $s$.
Such an embedding can be shown to have a regular neighborhood
that is the total space of an $(s-d)$-disc block bundle $\nu$ over $X$,
see e.g. \cite[Section 2]{banaglbundletransfer}.
This normal disc block bundle $\nu$ 
represents the stable vertical normal bundle of $\xi$
and can be taken to be a PL microbundle
(or PL $(\real^{s-d}, 0)$-bundle) since
$\BSPL \simeq \BBSPL$ for the stable classifying spaces.
(Such a stable vertical normal bundle exists even for
mock bundles with manifold blocks,
see \cite[IV.2, p. 83]{buonrs}.)
The image of the
$\MSPL$-cohomology Thom class of $\nu$ under the ring map
$\MSPL \to E$ yields an $E$-cohomology Thom class of $\nu$.
Cap product with this class defines the Thom homomorphism $\Phi$.

We shall consider the block transfer $\xi^!$ in the cases where
the ring spectrum $E$ is $\KO[\smlhf]$, $\syml (\rat)$
or $\syml (\rat)[\smlhf] = \syml (\intg)[\smlhf]$, and the ring maps
$\MSPL \to E$ are the orientations considered earlier.
The compatibility of these transfers will be established
in Lemma \ref{lem.kappacommtransf} and is essentially
a consequence of the multiplicativity of the map $\kappa$.
We need to be more precise about the involved Thom homomorphisms $\Phi$.
Our arguments involve the three Thom classes
discussed in Section \ref{sec.orientsullivanranicki}: The class
$u_\SPL (\alpha)$ in $\MSPL$-cohomology,  
$u_\syml (\alpha)$ in $\syml$-cohomology,
and the class $\Delta (\alpha)$ in $\KO[\smlhf]$-cohomology.
Let $E$ be a ring spectrum and let $m=s-d$ denote the rank 
of the aforementioned representative of the
stable vertical normal disc block bundle $\nu$.
The reduced $E$-cohomology group of the Thom space can be expressed 
as a relative group,
\[ \widetilde{E}^m (\Th (\nu)) \cong E^m (N,\partial N), \]
where $N$ is the total space of $\nu$ and
$\partial N$ the total space of the sphere bundle of $\nu$. 
Let
\[ \rho_*: E_* (N) \stackrel{\cong}{\longrightarrow} E_* (X) \]
be the inverse of the isomorphism induced on $E$-homology by the inclusion
$X \hookrightarrow N$ of the zero section.
If $\nu$ is $E$-orientable, then
using the cap product
\[ \cap: E^m (N,\partial N) \otimes E_n (N,\partial N) 
   \longrightarrow E_{n-m} (N)
   \stackrel{\rho_*}{\cong} E_{n-m} (X)  \]
with a Thom class ($E$-orientation) $u\in E^m (N,\partial N)$ for $\nu$,   
we obtain the \emph{Thom homomorphism}
\[   
\Phi := \rho_* (u \cap -):
\widetilde{E}_n (\Th (\nu)) \cong E_n (N,\partial N) \longrightarrow
 E_{n-m} (X).
\]
Since $\Delta (\nu)$ is a $\KO [\smlhf]$-cohomology orientation
of $\nu$ with $\Delta (\nu)=\Delta_* u_\SPL (\nu)$
(Definition \ref{def.deltaofplmicrobundle}), 
we get for the ring spectrum $E=\KO [\smlhf]$ the Thom homomorphism
\[   
\Phi = \rho_* (\Delta (\nu) \cap -):
\widetilde{\KO}_n (\Th (\nu))\otimes \ism \longrightarrow
 \KO_{n-m} (X)\otimes \ism.
\]
Similarly, since $u_\syml (\nu)$ is an $\syml [\smlhf]$-cohomology orientation
of $\nu$ with $u_\syml (\nu) = \sigma^* u_\SPL (\nu),$
we receive for the ring spectrum $E=\syml [\smlhf]$ the Thom homomorphism
\[   
\Phi = \rho_* (u_\syml (\nu) \cap -):
\widetilde{\syml}_n (\Th (\nu))\otimes \ism \longrightarrow
 \syml_{n-m} (X)\otimes \ism.
\]

\begin{lemma} \label{lem.thomcommkappa}
The Thom homomorphisms $\Phi$ 
on $\syml [\smlhf]$-homology and $\KO [\smlhf]$-homology
agree under the natural isomorphism $\kappa$, that is, the diagram
\[ \xymatrix{
\widetilde{\KO}_n (\Th (\nu))\otimes \ism \ar[r]^\Phi 
   \ar[d]_{\kappa_*}^\cong
  & \KO_{n-m} (X)\otimes \ism \ar[d]^{\kappa_*}_\cong \\
\widetilde{\syml}_n (\Th (\nu)) \otimes \ism \ar[r]^\Phi 
  & \syml_{n-m} (X)\otimes \ism
} \]
commutes.
\end{lemma}
\begin{proof}
As $\kappa_*$ is a natural transformation of homology theories,
it commutes with the isomorphism $\rho_*$.
Since $\kappa$ is a morphism of ring spectra, it respects cap products, i.e.
the diagram
\[ \xymatrix{
(\KO [\smlhf])^m (Y,A) \otimes (\KO [\smlhf])_n (Y,A) 
  \ar[r]^>>>>>\cap  \ar[d]_{\kappa_* \otimes \kappa_*}
  & (\KO [\smlhf])_{n-m} (Y) \ar[d]^{\kappa_*} \\
(\syml [\smlhf])^m (Y,A) \otimes (\syml [\smlhf])_n (Y,A) 
 \ar[r]^>>>>>\cap & (\syml [\smlhf])_{n-m} (Y)
} \]
commutes.
By Lemma \ref{lem.kappadeltaisranickiu},
$\kappa_* (\Delta (\nu)) = u_\syml (\nu).$
Therefore,
\begin{align*}
\kappa_* \Phi
&= \kappa_* \rho_* (\Delta (\nu) \cap -) = \rho_* \kappa_* (\Delta (\nu) \cap -) \\
&= \rho_* (\kappa_* (\Delta (\nu)) \cap \kappa_* (-)) 
   = \rho_* (u_\syml (\nu) \cap \kappa_* (-)) = \Phi \kappa_*.
\end{align*}
\end{proof}

\begin{remark}  \label{rem.phicommloc}
Since localization is multiplicative on the spectrum level, 
it takes the Thom homomorphism $\Phi$ defined by capping with 
the Thom class $u_\syml (\nu) \in (\widetilde{\syml \rat})^m (\Th (\nu))$
to the Thom homomorphism $\Phi$ defined by capping with
the localized class
$u_\syml (\nu) \in (\widetilde{\syml \rat} [\smlhf])^m (\Th (\nu))$, that is,
the diagram
\[ \xymatrix{
(\widetilde{\syml \rat})_n (\Th (\nu)) \ar[d] \ar[r]^\Phi &
    (\syml \rat)_{n-m} (X) \ar[d] \\
(\widetilde{\syml \rat} [\smlhf])_n (\Th (\nu)) \ar[r]^\Phi &
    (\syml \rat [\smlhf])_{n-m} (X)    
} \]
commutes. Thus $\syml$-theoretic transfers also
commute with localization away from $2$.
\end{remark}

Using the Thom homomorphisms $\Phi$ appearing in Lemma \ref{lem.thomcommkappa},
there are in particular transfers
\[ \xi^!: \syml_n (B)\otimes \ism \longrightarrow 
   \syml_{n+d} (X)\otimes \ism \]
and
\[ \xi^!: \KO_n (B)\otimes \ism \longrightarrow 
   \KO_{n+d} (X)\otimes \ism. \]   

\begin{lemma} \label{lem.kappacommtransf}
The diagram of transfers
\[ \xymatrix{
\KO_{*+d} (X) \otimes \ism \ar[r]^{\kappa_*}_\cong  & 
  \syml_{*+d} (X) \otimes \ism  \\
\KO_* (B) \otimes \ism \ar[u]^{\xi^!} \ar[r]^{\kappa_*}_\cong 
  & \syml_* (B) \otimes \ism \ar[u]_{\xi^!}
} \]
commutes.
\end{lemma}
\begin{proof}
By construction of the block transfer $\xi^!$, the diagram
factors as
\[ \xymatrix{
 \KO_{*+d} (X) \otimes \ism 
 \ar[r]^{\kappa_*}_\cong 
  & \syml_{*+d} (X) \otimes \ism \\
\widetilde{\KO}_{*+s} (\Th (\nu)) \otimes \ism 
 \ar[u]^\Phi \ar[r]^{\kappa_*}_\cong 
  & \widetilde{\syml}_{*+s} (\Th (\nu)) \otimes \ism 
    \ar[u]_\Phi \\
\widetilde{\KO}_{*+s} (S^s B^+) \otimes \ism 
 \ar[u]^{T(\xi)_*} \ar[r]^{\kappa_*}_\cong 
  & \widetilde{\syml}_{*+s} (S^s B^+) \otimes \ism 
    \ar[u]_{T(\xi)_*} \\
\KO_* (B) \otimes \ism \ar[u]^\sigma \ar[r]^{\kappa_*}_\cong 
  & \syml_* (B) \otimes \ism. \ar[u]_\sigma
} \]
The bottom and middle squares commute, as $\kappa_*$ is a natural transformation
of homology theories, while the top square involving the
Thom homomorphisms commutes
by Lemma \ref{lem.thomcommkappa}.
\end{proof}

The material on block bundle transfer homomorphisms 
developed above enables us to establish
our main result on bundle transfer of Siegel-Sullivan orientations:
\begin{thm} \label{thm.transfdeltaWitt}
If $\xi$ is an oriented PL $F$-block bundle with closed oriented 
PL manifold fiber $F$ over a closed Witt base $B$,
then the Siegel-Sullivan orientations of base and total space $X$
are related under block bundle transfer by
\[ \xi^! \Delta (B) = \Delta (X). \]
\end{thm}
\begin{proof}
Remark \ref{rem.phicommloc} implies that transfer commutes with localization
away from $2$:
the diagram
\[ \xymatrix{
(\syml \rat)_{*+d} (X) \ar[r]  & 
  (\syml \rat [\smlhf])_{*+d} (X) \\
(\syml \rat)_* (B) \ar[u]^{\xi^!} \ar[r]
  & (\syml \rat [\smlhf])_* (B) \ar[u]_{\xi^!}
} \]
commutes.
In \cite[Theorem 7.1]{banaglbundletransfer} we showed that the left hand transfer sends
$[B]_\syml$ to $[X]_\syml$.
Thus 
\[ \xi^! [B]_\syml = [X]_\syml \]
for the right hand transfer as well.
By (\ref{equ.kappadeltaissymlor}),
\[ 
\kappa_* \Delta (X) = [X]_\syml \in (\syml \rat [\smlhf])_{b+d} (X),~ 
\kappa_* \Delta (B) = [B]_\syml \in (\syml \rat [\smlhf])_b (B).
\]
Using Lemma \ref{lem.kappacommtransf},
\[ \kappa_* \xi^! \Delta (B)
  = \xi^! \kappa_* \Delta (B)
  = \xi^! [B]_\syml = [X]_\syml
  = \kappa_* \Delta (X). \]
It follows that $\xi^! \Delta (B) = \Delta (X)$, as $\kappa_*$ is
an isomorphism.  
\end{proof}

\section{Gysin Restriction of the Siegel-Sullivan Orientation}
\label{sec.gysinsiegelsullivan}

Our method based on the equivalence of $\mathbb{E}_\infty$-ring spectra
$\kappa: \KO [\smlhf] \simeq \syml (\real) [\smlhf]$ constructed in
Proposition \ref{prop.einftyequivkol},
together with results of \cite{banaglnyjm},
allows for a treatment of Gysin restrictions of
Siegel-Sullivan classes under normally nonsingular inclusions of singular
spaces in a fashion parallel to our analysis of bundle transfers
in the previous section. 
An inclusion $g:Y\hookrightarrow X$ of stratified spaces
is \emph{normally nonsingular} if $Y$ has an open tubular neighborhood that can be
equipped in a stratum preserving manner with the structure 
of a vector bundle $\nu$ over $Y$
such that $Y$ is identified with the zero section.
For example, the transverse intersection of a smooth submanifold
with a Whitney stratified set $X$ in an ambient smooth manifold is
normally nonsingular in $X$ (\cite[p. 47, Thm. 1.11]{gmsmt}).

Let $g:Y^{n-c} \hookrightarrow X^n$ be a codimension $c$ 
normally nonsingular inclusion
of closed Witt spaces with normal bundle $\nu$ of rank $c$.
Let $E$ be a ring spectrum such that $\nu$ has an $E$-orientation $u$.
We describe the Gysin restriction on $E$-homology associated
to $g$. The canonical map $j:X^+ \to \Th (\nu)$ induces a homomorphism
\[ j_*: E_* (X) \longrightarrow
        \widetilde{E}_* (\Th (\nu)). \]
As in the previous section, cap product with the $E$-orientation $u$ yields
the Thom homomorphism
\[ \Phi = \rho_* (u\cap -): \widetilde{E}_* (\Th (\nu))
    \longrightarrow E_{*-c} (Y). \]
Composition defines the \emph{Gysin restriction}
\[ g^! = \Phi \circ j_*:
       E_* (X) \longrightarrow E_{*-c} (Y). \]           
       
Now suppose that $\nu$ is $H\intg$-oriented, compatibly
with the orientations of $X$ and $Y$.       
Applying the above general description of $g^!$ to $E=\KO [\smlhf]$
with $u=\Delta (\nu)$, we obtain the Gysin homomorphism
\[ g^!: \KO_* (X) \otimes \ism 
          \longrightarrow \KO_{*-c} (Y) \otimes \ism, \]   
and applying it to $E=\syml [\smlhf]$
with $u=u_\syml (\nu)$, we obtain the Gysin homomorphism
\[ g^!: \syml_* (X) \otimes \ism 
          \longrightarrow \syml_{*-c} (Y) \otimes \ism. \]

\begin{lemma} \label{lem.kappacommgysin}
The diagram of Gysin restrictions
\[ \xymatrix{
\KO_{*-c} (Y) \otimes \ism \ar[r]^{\kappa_*}_\cong  & 
  \syml_{*-c} (Y) \otimes \ism  \\
\KO_* (X) \otimes \ism \ar[u]^{g^!} \ar[r]^{\kappa_*}_\cong 
  & \syml_* (X) \otimes \ism \ar[u]_{g^!}
} \]
commutes.
\end{lemma}
\begin{proof}
By construction of the restrictions $g^!$, the diagram
factors as
\[ \xymatrix{
 \KO_{*-c} (Y) \otimes \ism 
 \ar[r]^{\kappa_*}_\cong 
  & \syml_{*-c} (Y) \otimes \ism \\
\widetilde{\KO}_{*} (\Th (\nu)) \otimes \ism 
 \ar[u]^\Phi \ar[r]^{\kappa_*}_\cong 
  & \widetilde{\syml}_{*} (\Th (\nu)) \otimes \ism 
    \ar[u]_\Phi \\
\KO_{*} (X) \otimes \ism 
 \ar[u]^{j_*} \ar[r]^{\kappa_*}_\cong 
  & \syml_{*} (X) \otimes \ism. 
    \ar[u]_{j_*} 
} \]
The bottom square commutes, as $\kappa_*$ is a natural transformation
of homology theories, while the top square involving the
Thom homomorphisms commutes
by Lemma \ref{lem.thomcommkappa}.
\end{proof}

The Siegel-Sullivan orientation behaves under normally
nonsingular Gysin restrictions as follows.
\begin{thm} \label{thm.gysindeltaWitt}
Let $g:Y^{n-c} \hookrightarrow X^n$ be an oriented normally nonsingular inclusion
of closed Witt spaces. The $\KO [\smlhf]$-homology Gysin map $g^!$ of $g$
sends the Siegel-Sullivan orientation of $X$ to the Siegel-Sullivan
orientation of $Y$:
\[ g^! \Delta (X) = \Delta (Y). \]
\end{thm}
\begin{proof}
The proof is analogous to the one of Theorem \ref{thm.transfdeltaWitt}.
Remark \ref{rem.phicommloc}
implies that Gysin restriction commutes with localization
away from $2$:
the diagram
\[ \xymatrix{
(\syml \rat)_{*-c} (Y) \ar[r]  & 
  (\syml \rat [\smlhf])_{*-c} (Y) \\
(\syml \rat)_* (X) \ar[u]^{g^!} \ar[r]
  & (\syml \rat [\smlhf])_* (X) \ar[u]_{g^!}
} \]
commutes.
In \cite[Theorem 3.17]{banaglnyjm} we showed that the left hand restriction sends
$[X]_\syml$ to $[Y]_\syml$.
Thus 
\[ g^! [X]_\syml = [Y]_\syml \]
for the right hand restriction as well.
By (\ref{equ.kappadeltaissymlor}),
\[ 
\kappa_* \Delta (X) = [X]_\syml \in (\syml \rat [\smlhf])_n (X),~ 
\kappa_* \Delta (Y) = [Y]_\syml \in (\syml \rat [\smlhf])_{n-c} (Y).
\]
Using Lemma \ref{lem.kappacommgysin},
\[ \kappa_* g^! \Delta (X)
  = g^! \kappa_* \Delta (X)
  = g^! [X]_\syml = [Y]_\syml
  = \kappa_* \Delta (Y). \]
It follows that $g^! \Delta (X) = \Delta (Y)$, as $\kappa_*$ is
an isomorphism.  
\end{proof}

In tandem, Theorems \ref{thm.gysindeltaWitt} and \ref{thm.transfdeltaWitt}
show that the transfer associated to a
normally nonsingular map $Y\to B$ (Goresky-MacPherson \cite[5.4.3]{gmih2}, 
Fulton-MacPherson \cite{fultonmacpherson})
sends $\Delta (B)$ to $\Delta (Y)$.

\end{document}